\documentclass[11pt, a4paper]{amsart}

\usepackage{amsfonts,amsmath,amssymb, amscd}

\usepackage[all]{xy}

\newtheorem{theorem}{Theorem}[section]
\newtheorem{lemma}[theorem]{Lemma}
\newtheorem{prop}[theorem]{Proposition}
\newtheorem{corollary}[theorem]{Corollary}

\theoremstyle{definition}
\newtheorem{definition}[theorem]{Definition}
\newtheorem{rem}[theorem]{Remark}

\newtheorem{example}[theorem]{Example}

\usepackage{color}

\newcommand\pf{\begin{proof}}
\newcommand\epf{\end{proof}}

\newcommand{\mr}{\mathrm}

\numberwithin{equation}{section}

\title[On functor points of affine supergroups]
{On functor points of affine supergroups}

\author[A.~Masuoka]{Akira Masuoka}
\address{Akira Masuoka,
Institute of Mathematics, 
University of Tsukuba, 
Ibaraki 305-8571, Japan}
\email{akira@math.tsukuba.ac.jp}

\author[T.~Shibata]{Taiki Shibata}
\address{Taiki Shibata, 
Department of Applied Mathematics,
Okayama University of Science,
1-1 Ridaicho, Kita-ku, Okayama
700-0005, Japan}
\email{shibata@xmath.ous.ac.jp}

\dedicatory{Dedicated to Professor Susan Montgomery in honor of her distinguished career}

\begin{document}

\begin{abstract}
To construct an affine supergroup from a Harish-Chandra pair, Gavarini \cite{G} 
invented a natural method, which first constructs a group functor and then proves
that it is representable.
We give a simpler and more conceptual presentation of his construction
in a generalized situation, using
Hopf superalgebras over a superalgebra.
As an application of the construction,    
given a closed super-subgroup of an algebraic supergroup,
we describe the normalizer and the centralizer, 
using Harish-Chandra pairs. We also prove a tensor product decomposition theorem for 
Hopf superalgebras, and describe explicitly by cocycle deformation, the 
difference which results from the two choices of dualities
found in literature.
\end{abstract}

\maketitle

\noindent
{\sc Key Words:}
affine supergroup, algebraic supergroup, 
Hopf superalgebra, Harish-Chandra pair

\medskip
\noindent
{\sc Mathematics Subject Classification (2000):}
14L15, 
14M30, 
16T05  

\section{Introduction}\label{sec:introduction}
\subsection{Basic definitions}\label{subsec:definitions}
In this paper we work over a non-zero commutative ring $\Bbbk$. 
The unadorned $\otimes$ is the tensor product over $\Bbbk$. A $\Bbbk$-module
is said to be $\Bbbk$-\emph{finite}, if it is finitely generated.   

The word ``super" is a synonym of ``graded by the group $\mathbb{Z}_2=\{ 0, 1\}".$
Therefore, a $\Bbbk$-\emph{supermodule} is a $\Bbbk$-module $V$ graded by $\mathbb{Z}_2$
so that $V =V_0 \oplus V_1$. When we say that $v$ is an element of $V$, we assume
that it is homogeneous, and denote its degree by $|v|$. If $i=|v|$, the element or the component
$V_i$ is said to be \emph{even} or \emph{odd}, according to  $i=0$ or $i=1$. We say that $V$ is \emph{purely 
even} if $V=V_0$, and is \emph{purely odd} if $V=V_1$. 
The dual $\Bbbk$-module $V^*$ of $V$ is again a $\Bbbk$-supermodule so that
$(V^*)_i=V_i^*$, $i=0,1$.  

The $\Bbbk$-supermodules form a symmetric tensor category $\mathsf{SMod}_{\Bbbk}$ 
with respect to the tensor product $\otimes$, the unit object $\Bbbk$ and the super-symmetry
\begin{equation}\label{eq:super-symmetry}
c_{V, W} : V \otimes W \to W \otimes V,\quad c_{V,W}(v \otimes w) = (-1)^{|v||w|} w \otimes v. 
\end{equation}
Super-objects are objects, such as algebra object or Hopf-algebra object, 
defined in $\mathsf{SMod}_{\Bbbk}$. They are called with ``super" attached, so as
(\emph{Hopf}) \emph{superalgebras}. Ordinary objects, such
as (Hopf) algebras, are regarded as purely even super-objects. We let
\[
\mathsf{SAlg}_{\Bbbk}, \quad \mathsf{Alg}_{\Bbbk}
\]
denote the category of super-commutative superalgebras
and its full subcategory
consisting of all commutative algebras, respectively. 
A \emph{group functor} is a group-valued functor
defined on $\mathsf{SAlg}_{\Bbbk}$ or $\mathsf{Alg}_{\Bbbk}$.

Working over an arbitrary commutative ring, we should be careful to define super-commutativity.
By saying that a superalgebra $A$ is \emph{super-commutative}, we require that $a^2=0$ for all 
odd elements $a \in A_1$, in addition to the usual requirement that the product on $A$
should be invariant, composed with the super-symmetry.

\subsection{Algebraic supergroups and Harish-Chandra pairs}\label{subsec:ASG_and_HCP}
The notion of affine or algebraic groups defined in \cite[Part I, 2.1]{J} 
is directly generalized to the super context, as follows. An \emph{affine supergroup}
is a representable group functor $\mathbf{G}$ defined on $\mathsf{SAlg}_{\Bbbk}$.
The super-commutative superalgebra $\mathcal{O}(\mathbf{G})$ which represents $\mathbf{G}$
necessarily has a Hopf superalgebra structure which arises uniquely from 
the group structure on $\mathbf{G}$. 
An affine supergroup $\mathbf{G}$ is called an \emph{algebraic supergroup} if
$\mathcal{O}(\mathbf{G})$ is finitely generated. 

We let $\mathsf{ASG}_{\Bbbk}$ denote the category of algebraic supergroups. Given
$\mathbf{G}\in \mathsf{ASG}_{\Bbbk}$, 
we have
\[
G=\mathbf{G}_{ev}\ \text{and}\ \mathfrak{g}=\operatorname{Lie}(\mathbf{G}),
\]
where $G$ is the algebraic group defined as a restricted group functor by $G = \mathbf{G}|_{\mathsf{Alg}_{\Bbbk}}$,
and $\mathfrak{g}$ is the Lie superalgebra of $\mathbf{G}$.
Very roughly speaking, a \emph{Harish-Chandra pair} is such a pair $(G, \mathfrak{g})$ with
some additional structures involved; the notion is due to Kostant \cite{Kostant}, and 
an irredundant definition will be reproduced from   
\cite{M3, MZ} in Section \ref{subsec:equiv_over_fields} when $\Bbbk$ is a field of 
characteristic $\ne 2$. In this last situation the last cited articles reformulated
the result \cite[Theorem 29]{M2}, which is formulated in purely Hopf-algebraic terms, 
so that $\mathbf{G}\to (\mathbf{G}_{ev}, \operatorname{Lie}(\mathbf{G}))$ gives a
category equivalence
\begin{equation}\label{eq:equiv_intro}
\mathsf{ASG}_{\Bbbk}\overset{\approx}{\longrightarrow} \mathsf{HCP}_{\Bbbk},
\end{equation}
where $\mathsf{HCP}_{\Bbbk}$ denotes the category of Harish-Chandra pairs; the result was applied
in several papers including \cite{M2, M3, MS, GZ, MZ}.
An analogous category equivalence for super Lie groups had been proved 30 years before by Koszul \cite{Koszul}
in the $C^{\infty}$ situation, and slightly before
by Vishnyakova \cite{Vi} in the complex analytic situation; see
also \cite[Section 7.4]{CCF}. 
Very recently, 
the same result was proved by Hoshi and the first-named author \cite{HM}
for super Lie groups over a complete field of characteristic $\ne 2$. 
For the above cited result \cite[Theorem 29]{M2} in the algebraic situation, 
crucial is the result \cite[Theorem 4.5]{M1} that for every $\mathbf{G} \in \mathsf{ASG}_{\Bbbk}$,
the Hopf superalgebra $\mathcal{O}(\mathbf{G})$ decomposes so that 
\begin{equation}\label{eq:Int_TPD}   
\mathcal{O}(\mathbf{G}) \simeq \mathcal{O}(\mathbf{G}_{ev}) \otimes \wedge(\mathfrak{g}_1^*) 
\end{equation}
as a left $\mathcal{O}(\mathbf{G}_{ev})$-comodule superalgebra with counit, 
where $\mathfrak{g}=\operatorname{Lie}(\mathbf{G})$. 
In the previous work \cite{MS} the authors proved the category equivalence above when $\Bbbk$ is a commutative ring
which is $2$-torsion free, restricting the objects of $\mathsf{ASG}_{\Bbbk}$ to those $\mathbf{G}$ such that
$\mathcal{O}(\mathbf{G})$ decomposes as above. 
Gavarini \cite{G} proved the same result more generally when $\Bbbk$ is an arbitrary commutative ring,
requiring Lie superalgebras to have an additional structure, $2$-\emph{operations}. 

The quasi-inverse of the functor \eqref{eq:equiv_intro} constructed in \cite{M2, MS} is close 
to the one constructed by Koszul and others for super Lie groups. 
Gavarini's quasi-inverse, which is our concern, is 
new and natural; it first constructs a group functor and then proves
that it is representable. He defines the groups of functor points by generators and relations.

We remark that in the appendix of \cite{MS}, Gavarini's category equivalence was re-proved 
by using the method of the cited
article, to supplement the original proof which skipped some necessary arguments. 

\subsection{Three purposes}\label{subsec:purposes}
This paper is written for three purposes.

The first one is to give a simpler and more conceptual 
presentation of Gavarini's construction cited above. Given a Harish-Chandra pair $(G, \mathfrak{g})$
and $A \in \mathsf{SAlg}_{\Bbbk}$, we construct the group $\mathbf{\Gamma}(A)$ which shall be the functor
points of the desired affine supergroup $\mathbf{\Gamma}$, in two steps. Recall that  
the universal envelope $\mathbf{U}(\mathfrak{g})$ of $\mathfrak{g}$ is naturally a
Hopf superalgebra, 
whence the base extension $A \otimes \mathbf{U}(\mathfrak{g})$ to $A$ is 
the Hopf superalgebra over $A$. Consider the following elements in this last Hopf superalgebra over $A$:
\[ 
e(a, v) = 1 \otimes 1 + a \otimes v,\quad f(\epsilon, x)= 1\otimes 1 + \epsilon \otimes x, 
\]
where $a \in A_1$,\ $v\in \mathfrak{g}_1$,\ $x \in \mathfrak{g}_0$, and $\epsilon \in A_0$ with $\epsilon^2=0$.
Each of these elements, 
being the sum of the identity element and an even 
primitive whose tensor-square is zero, is an even grouplike; see Lemma \ref{lem:e(a,v)}.
The first step of our construction is to construct the group $\mathbf{\Sigma}(A)$ 
generated by these even grouplikes. 
The subgroup $F(A_0)$ generated by all $f(\epsilon, x)$, being included in $A_0 \otimes U(\mathfrak{g}_0)$,
has a natural group map into 
the group $G(A_0)$ of $A_0$-points of $G$. 
The second step is to construct the desired $\mathbf{\Gamma}(A)$ from
$\mathbf{\Sigma}(A)$ by what we call the \emph{base extension} along the group map $F(A_0) \to G(A_0)$;
this notion of base extensions of groups is defined and discussed in Section \ref{sec:base_extension}. 
Our construction will be seen to be natural, when one thinks of the natural pairing of 
$A \otimes \mathbf{U}(\mathfrak{g})$ with $\mathcal{O}(\mathbf{\Gamma})$; see Lemma \ref{lem:group_map}. 
The construction is done in Section \ref{sec:construction}, 
in a more generalized situation, aiming 
at an application 
to super Lie groups; see Remark \ref{rem:super_Lie_group}. 

We have described the contents of Sections \ref{sec:base_extension} and \ref{sec:construction}. 
Section \ref{sec:preliminary} is devoted to discussing some basic results on super-objects that include
the comparison of dualities explained in the next subsection. 

Section \ref{sec:equivalence}
starts with the subsection in which we
re-prove Gavarini's category equivalence 
cited above, using our method of construction.
This aims to supplement again Gavarini's original proof;
see Remark \ref{rem:compare_with_G}. 
Recall from the second paragraph of Section \ref{subsec:ASG_and_HCP} that 
the algebraic supergroups $\mathbf{G}$ over an arbitrary commutative ring $\Bbbk$, 
for which the category equivalence will be re-proven, 
are assumed so that $\mathcal{O}(\mathbf{G})$
decomposes as in \eqref{eq:Int_TPD}. The second purpose of ours, which is achieved
in Section \ref{subsec:TPD}, is to prove, along the line of our renewed proof, that the assumption above is 
necessarily satisfied
if $\mathcal{O}(\mathbf{G}_{ev})$ is $\Bbbk$-flat; see Theorem \ref{thm:TPD}.
This theorem benefits two results cited in the second paragraph of Section \ref{subsec:ASG_and_HCP}. 
In fact it improves our previous result in \cite{MS},
in which $\Bbbk$ is assumed to be $2$-torsion free, while
the proof gives an alternative proof of \cite[Theorem 4.5]{M1}, 
in which $\Bbbk$ is assumed to be a field;
see Remark \ref{rem:compare_TPD}.

In the final Section \ref{sec:equiv_over_fields}, which consists of 2 subsections,
we suppose that $\Bbbk$ is a field of characteristic $\ne 2$. 
In Section \ref{subsec:equiv_over_fields} we 
reproduce the category equivalence \eqref{eq:equiv_intro} from \cite{M3, MZ},
giving the irredundant definition of Harish-Chandra pairs referred to above. 
In Section \ref{subsec:application}, 
given an algebraic supergroup $\mathbf{G}$ and its closed super-subgroup $\mathbf{H}$, we describe
the normalizer $\mathcal{N}_{\mathbf{G}}(\mathbf{H})$ and the centralizer $\mathcal{Z}_{\mathbf{G}}(\mathbf{H})$
in terms of Harish-Chandra pairs; see Theorem \ref{thm:normal_central}. 
Gavarini's construction is quite useful to discuss group-theoretical properties of affine supergroups,
and it applies to prove the last cited theorem. This application is indeed the third purpose of ours. 

\subsection{Comparing dualities}\label{subsec:duaity_in_addition}
We will make it clear that there are two choices, 
when we discuss the duality of Hopf superalgebras; this may not have been clearly recognized
so far.     
If the simpler duality is chosen,  as was done by \cite{M2, M3, MS, MZ}, 
one defines a pairing $\langle \ , \ \rangle : \wedge(W^*) \times  \wedge(W) \to \Bbbk$
between the exterior algebras on a $\Bbbk$-finite free module $W$ and on its dual $W^*$, 
as usually so that
\begin{equation}\label{eq:usual_duality} 
\langle v_1\wedge \dots \wedge v_n ,\ w_1\wedge \dots \wedge w_m \rangle
= \delta_{n,m} \, \det\big( v_i(w_j) \big),\quad m, n \ge 0, 
\end{equation}
where $v_i \in W^*$, $w_i \in W$. 
We have to choose
the other one, replacing \eqref{eq:usual_duality} with \eqref{eq:duality} below, 
since the simpler duality does not work well for Hopf superalgebras
over $A\in \mathsf{SAlg}_{\Bbbk}$, in general. In Section \ref{subsec:compare} this circumstance is explained, and
the difference caused by choices is described in terms of \emph{cocycle deformations}.
Fortunately, the category equivalence 
\eqref{eq:equiv_intro} obtained with our choice of duality
coincides with the one obtained before in \cite{M3, MS, MZ}, up to an involutive category isomorphism 
$\mathsf{HCP}_{\Bbbk}\to \mathsf{HCP}_{\Bbbk}$, as will be seen 
in Remarks \ref{rem:compare_with_MS} and \ref{rem:coincidence_up_to_isom}. 

\section{Base extension of groups}\label{sec:base_extension}

Suppose that the quintuple
\[ (\Sigma,\ F,\ G,\ i,\ \alpha) \]
consists of groups $\Sigma$, $F$ and $G$, a group map $i : F \to G$, and anti-group map 
$\alpha : G \to \operatorname{Aut}(\Sigma)$ such that
 \begin{itemize}
\item[(A1)] $F$ is a subgroup of $\Sigma$, 
\item[(A2)] $\varphi^{i(f)} = f^{-1} \varphi f$ for all $f \in F$, $\varphi \in \Sigma$,
\item[(A3)] $f^g \in F$ and $i(f^g)=g^{-1}i(f)g$, 
\end{itemize}
where $f \in F$,\ $g \in G$,\ $\varphi \in \Sigma$, and $\varphi^g$ stands for $\alpha(g)(\varphi)$.
Suppose that $F$ and $G$ act on $\Sigma$ and $G$, respectively, from the right by inner automorphisms. 
Then (A2) reads that $i$ preserves the actions on $\Sigma$, while (A3) reads that $F$ is $G$-stable,
and $i$ is $G$-equivariant. 

Let $G \ltimes \Sigma$ be the semi-direct product given by $\alpha$, and set 
\[ \Xi = \{ (i(f), f^{-1}) \in G \ltimes \Sigma \mid f \in F \}.   \]
Then one sees from (A2)--(A3) that $\Xi$ is a normal subgroup of $G \ltimes \Sigma$; in particular,
$\Sigma$ centralizes $\Xi$.
We let
\[ \Gamma = \Gamma(\Sigma, F, G, i, \alpha) \]
denote the quotient group $G \ltimes \Sigma/\Xi$. 

\begin{lemma}\label{lem:base_extension}
We have the following.
\begin{itemize}
\item[(1)] The composite $G \to G \ltimes \Sigma\to \Gamma$ of the inclusion with the quotient map
is an injection, through which we will regard $G$ as a subgroup of $\Gamma$. 
\item[(2)] The composite $\Sigma \to G \ltimes \Sigma\to \Gamma$ of the inclusion with the quotient map
induces a bijection $F\backslash \Sigma \to G\backslash \Gamma$ between the sets of right cosets. 
\end{itemize}
\end{lemma}
\begin{proof}
Choose arbitrarily a set $X\subset F$ of representatives of $F \backslash \Sigma$. Then the
product map $p : F \times X \to \Sigma$, $p(f,x) = fx$ is a bijection, through which we will identify
$\Sigma$ with $F \times X$. Then we have $G \ltimes \Sigma= (G \ltimes F) \times X$ as left $G \ltimes F$-sets. 
Note $\Xi \subset G\ltimes F$ and that the canonical map 
$G \to G\ltimes F/\Xi=\Xi\backslash G\ltimes F$ is an isomorphism. The direct product with $\mr{id}_X$
gives a left $G$-equivariant bijection,  
$q : G \times X \to (\Xi\backslash G\ltimes F) \times X = \Gamma$. 
The injectivity of $q$ yields Part 1, since one may choose $X$ so as containing the identity element.  
The equivariant bijections $p$ and $q$ induce bijections, 
$\overline{p}: X \to F\backslash \Sigma$,\  
$\overline{q}: X \to G\backslash \Gamma$. 
We see that $\overline{q}\circ \overline{p}^{-1} : 
F\backslash \Sigma \to G\backslash \Gamma$ is the bijection claimed by Part 2.
\end{proof}

Taking into account the property shown in Part 2 above we say:

\begin{definition}\label{def:base_extension}
$\Gamma$ is the \emph{base extension} of $\Sigma$ along $i : F \to (G,\alpha)$.
Here we suppose that $i$ is a morphism of groups acting on $\Sigma$,
bearing in mind the action of $F$ by inner automorphisms.  
\end{definition}

\section{Basic results on super-objects}\label{sec:preliminary}

In what follows we work over a non-zero commutative ring $\Bbbk$. This $\Bbbk$ is supposed
to be arbitrary unless otherwise specified. 

\subsection{Pairings}\label{subsec:pairing}
Recall from Section \ref{subsec:definitions} that $\mathsf{SAlg}_{\Bbbk}$ denotes 
the category of those superalgebras $A$ which are \emph{super-commutative} in the sense
that $A_0$ is central in $A$, and $a^2=0$ for all $a \in A_1$; see \cite[Section 2.1.1]{G}, for example. 
Let $A\in \mathsf{SAlg}_{\Bbbk}$.
An $A$-\emph{supermodule} is a left $A$-module object in $\mathsf{SMod}_{\Bbbk}$; this is
identified with the right $A$-module object which is defined on the same $\Bbbk$-supermodule, say $M$, 
by $ma:=(-1)^{|a||m|}am$, 
$a \in A$, $m\in M$. Given $A$-supermodules $M$, $N$, let 
$M\otimes_A N$ denote the quotient $\Bbbk$-supermodule of $M \otimes N$ defined by the relations
\[ ma \otimes n = m\otimes an, \quad a\in A,\ m\in M,\ n \in N. \]
This is naturally an $A$-supermodule. The $A$-supermodules form a symmetric tensor category
$A\text{-}\mathsf{SMod}$, where the tensor product is the $\otimes_A$ just defined, and the unit object is $A$.
The symmetry is the one induced from the super-symmetry $c_{M,N}$ (see \eqref{eq:super-symmetry}), 
and it will be denoted
by the same symbol. A Hopf superalgebra over $A$
is a Hopf-algebra object in $A\text{-}\mathsf{SMod}$. The structure maps
of a Hopf superalgebra 
$\mathcal{H}$ over $A$ will be denoted by
\[ 
\Delta : \mathcal{H} \to \mathcal{H} \otimes_A \mathcal{H},\; \Delta(h)=h_{(1)}\otimes h_{(2)},\quad 
\varepsilon : \mathcal{H}\to A,\quad 
\mathcal{S} : \mathcal{H}\to \mathcal{H}. 
\]

A \emph{pairing} between objects $M$ and $N$ in $A\text{-}\mathsf{SMod}$ is a morphism 
$M\otimes_A N \to A$ in $A\text{-}\mathsf{SMod}$, 
which will be often presented as $\langle \ , \ \rangle : M \times N \to A$,
$\langle m,\ n \rangle =$~the value of~$m \otimes n$. The \emph{tensor product} with
another pairing $\langle \ , \ \rangle : M' \otimes_A N' \to A$ is the pairing between $M \otimes_A M'$ and
$N \otimes_A N'$ which is defined to be the composite
\[ 
(M \otimes_A M')\otimes_A (N \otimes_A N')
\to
(M\otimes_A N)\otimes_A (M' \otimes_A N')
\to
A\otimes_A A = A 
\]
of $\mathrm{id}_M \otimes_A c_{M',N} \otimes_A \mathrm{id}_{N'}$ with 
$\langle \ , \ \rangle \otimes_A \langle \ , \ \rangle$. Explicitly, it is defined by
\begin{equation}\label{eq:tensor_product_pairing}
\langle m\otimes m',\ n \otimes n'\rangle = (-1)^{|m'||n|} \langle m,\ n\rangle \, \langle m',\ n' \rangle, 
\end{equation}
where $m \in M$, $m'\in M'$, $n \in N$, $n' \in N'$.
We remark that if $A = \Bbbk$, then
the sign $(-1)^{|m'||n|}$ above can be replaced by either  
$(-1)^{|m||n'|}$, $(-1)^{|m||m'|}$ or $(-1)^{|n||n'|}$. 
  
Let $\mathcal{L}$, $\mathcal{H}$ be Hopf superalgebras over $A$. A pairing
$\langle \ , \ \rangle : \mathcal{L}\times \mathcal{H}\to A$ is called a \emph{Hopf pairing}, if we have
\begin{align}
\langle x,\ hk \rangle = \langle \Delta(x),\ h \otimes k\rangle, \quad 
&\langle xy,\ h \rangle =\langle x \otimes y,\ \Delta(h) \rangle, \label{eq:xyhk} \\
\langle x,\ 1 \rangle = \varepsilon(x), \quad & \langle 1,\ h \rangle = \varepsilon(h), \label{eq:x11h}
\end{align}
where $x,y \in \mathcal{L}$,\ $h, k \in \mathcal{H}$. On the right-hand sides of \eqref{eq:xyhk}
appears the tensor product of two copies of the pairing.
The conditions imply
\[ 
\langle \mathcal{S}(x),\ h \rangle = \langle x,\ \mathcal{S}(h) \rangle,\quad 
x \in \mathcal{L},\ h \in \mathcal{H}. 
\]

Just as in the non-super situation, the set  
\[
\mathsf{Gpl}(\mathcal{L}):= \{ g \in \mathcal{L}_0 \mid 
\Delta(g)=g\otimes_A  g,\ \varepsilon(g)=1 \} 
\]
of all even grouplikes in $\mathcal{L}$
is a group under the product of $\mathcal{L}$, and the set 
\[ 
\mathsf{SAlg}_A(\mathcal{H}, A)
\]
of all superalgebra maps $\mathcal{H}\to A$ over $A$
is a group under the convolution product \cite[Page 6]{Mo}. 
Our construction of affine supergroups
is inspired by the following simple fact, which is easy to prove. 

\begin{lemma}\label{lem:group_map}
A Hopf pairing $\langle \ , \ \rangle : \mathcal{L}\times \mathcal{H}\to A$ induces the group map 
\[ \mathsf{Gpl}(\mathcal{L})\to \mathsf{SAlg}_A(\mathcal{H}, A),\quad 
g \mapsto \langle g,\ -\rangle. \]
\end{lemma}

Here is a typical example of Hopf pairings over $\Bbbk$. 

\begin{example}[\text{cf. \eqref{eq:usual_duality},\ \cite[Eq.~(5)]{M2}}]\label{ex:exterior}
Let $W$ be a $\Bbbk$-module which is $\Bbbk$-finite free. We regard the exterior algebra
$\wedge(W)$ on $W$ as a (super-commutative)
Hopf superalgebra in which every element in $W$ is a (square-zero) odd primitive. 
We have another such Hopf superalgebra $\wedge(W^*)$. A Hopf pairing 
$\langle \ , \ \rangle : \wedge(W^*) \times \wedge(W) \to \Bbbk$ is defined by
\begin{equation}\label{eq:duality} 
\langle v_1\wedge \dots \wedge v_n ,\ w_1\wedge \dots \wedge w_m \rangle
= \delta_{n,m} \, (-1)^{\binom{n}{2}}\, \det\big( v_i(w_j) \big),\quad m, n \ge 0, 
\end{equation}
where $v_i \in W^*$, $w_i \in W$. By convention we have
$\binom{0}{2} = \binom{1}{2}=0$. 
\end{example}

\subsection{Comparing dualities}\label{subsec:compare}
In the situation above we suppose $A = \Bbbk$, and consider 
super-objects and pairings over $\Bbbk$. 

Suppose that $\mathcal{H}$ is a super-coalgebra. Then we make the dual $\Bbbk$-supermodule
$\mathcal{H}^*$ uniquely 
into a superalgebra so that the canonical pairing $\mathcal{H}^* \times \mathcal{H} \to \Bbbk$
satisfies the second equations of \eqref{eq:xyhk}, \eqref{eq:x11h}. This is the same as saying that
the pairing
$\mathcal{H} \times \mathcal{H}^* \to \Bbbk$, with the sides switched, 
satisfies the first equations of \eqref{eq:xyhk}, \eqref{eq:x11h}.
The identity of $\mathcal{H}^*$ is the counit of $\mathcal{H}$, and the product is given by 
\[ pq(h) = (-1)^{|p||q|}\, p(h_{(1)})\, q(h_{(2)}), \quad p, q \in \mathcal{H}^*,\ h \in \mathcal{H}. \]
We denote this superalgebra by 
$\mathcal{H}^{\bar{*}}.$

Similarly, if $\mathcal{H}$ is Hopf superalgebra which is $\Bbbk$-finite projective, we make $\mathcal{H}^*$
uniquely into a Hopf superalgebra denoted by $\mathcal{H}^{\bar{*}}$, so that
$\mathcal{H}^* \times \mathcal{H} \to \Bbbk$ or $\mathcal{H} \times \mathcal{H}^* \to \Bbbk$ 
is a Hopf pairing. We call $\mathcal{H}^{\bar{*}}$ the \emph{dual Hopf superalgebra} of $\mathcal{H}$. 
Since the Hopf pairing given in Example \ref{ex:exterior} is non-degenerate, it follows that
the Hopf superalgebras $\wedge(W)$ and $\wedge(W^*)$ are dual to each other. 

Suppose that
\[
\langle \ , \ \rangle : V \times W \to \Bbbk,\quad \langle \ , \ \rangle : V' \times W' \to \Bbbk
\]
are pairings over $\Bbbk$. 
We remark that in the articles \cite{M2, M3, MS}, 
the tensor product of pairings is supposed to be the ordinary one
\[
\langle v\otimes w,\ v'\otimes w'\rangle_{\text{ord}} = \langle v,\ w \rangle\, \langle v',\ w' \rangle,
\]
just as in the non-super situation. 
This is justified since it holds that
$\langle \ , \ \rangle_{\text{ord}} \circ (c_{V,W}\otimes\, \operatorname{id}_{W'\otimes V'}) 
= \langle \ , \ \rangle_{\text{ord}} \circ (\operatorname{id}_{V \otimes W}\otimes\, c_{W',V'})$;
see the proof of \cite[Corollary 3]{M2} and the following remark. 
Over $A \in \mathsf{SAlg}_{\Bbbk}$ in general, this 
is not true any more. Therefore, we chose the definition as in \eqref{eq:tensor_product_pairing},
so that we have
$\langle \ , \ \rangle \circ (c_{M,N} \otimes_A \operatorname{id}_{N'\otimes_A M'})
= \langle \ , \ \rangle \circ (\operatorname{id}_{M \otimes_A N}\otimes_A\, c_{N',M'})$, indeed.
Due to these different choices, the Hopf pairing given by \eqref{eq:duality}
is different from the ordinary one given by \eqref{eq:usual_duality} or \cite[Eq.~(5)]{M2}. 
Note also that the dual (Hopf) superalgebras given above 
are different from those given in the cited articles. 
We are going to clarify this difference.

Let $\Bbbk^{\times}$ denote the multiplicative group of all units in $\Bbbk$, and regard it as a trivial
module over the group $\mathbb{Z}_2 =\{0,1 \}$. Then the map
$\sigma : \mathbb{Z}_2\times \mathbb{Z}_2 \to \Bbbk^{\times}$ defined by
\[
\sigma(i,j) = (-1)^{ij},\quad i, j \in \mathbb{Z}_2
\]
is a $2$-cocycle. Therefore, the identity functor 
\[
\mathsf{SMod}_{\Bbbk} \to \mathsf{SMod}_{\Bbbk},\quad V \mapsto V={}_{\sigma}V
\]
together with the tensor structure
\begin{align}\label{eq:tensor_structure}
{}_{\sigma}V\otimes {}_{\sigma}W \to~~&{}_{\sigma}(V \otimes W),\; 
v \otimes w \mapsto \sigma(|v|,|w|)\, v \otimes w,\\
& \mathrm{id} : \Bbbk \to \Bbbk = {}_{\sigma}\Bbbk \notag
\end{align} 
form a tensor equivalence. One sees that this preserves the super-symmetry, and it is an involution since
$\sigma(i,j)^2=1$,\ $i, j \in \mathbb{Z}_2$.  
It follows that if $\mathcal{H}$ is a super-object, e.g. 
a Hopf superalgebra, over $\Bbbk$, then ${}_{\sigma}\mathcal{H}$ is such an object, and 
${}_{\sigma}({}_{\sigma}\mathcal{H})=\mathcal{H}$. This ${}_{\sigma}\mathcal{H}$ is called the
(\emph{cocycle}) \emph{deformation} of $\mathcal{H}$ by $\sigma$; see \cite[Section 1.1]{M}, for example. 

If $\Bbbk$ contains a square root $\sqrt{-1}$ of $-1$, then $\sigma$ is the coboundary of
\[ \nu : \mathbb{Z}_2 \to \Bbbk^{\times}, \quad \nu(0) = 1,\ \nu(1)= \sqrt{-1}. \] 
It follows that
\[ {}_{\sigma}V \mapsto V, \quad v \mapsto \nu(|v|)\, v \]
gives a natural isomorphism from the tensor equivalence ${}_{\sigma}(\ )$ given by $\sigma$ to
the identity tensor functor, whence the deformation ${}_{\sigma}\mathcal{H}$ by $\sigma$
is naturally isomorphic to the original $\mathcal{H}$, in this case. 

Given two pairings over $\Bbbk$ as above, we have
\begin{align*}
\langle v\otimes w,\ v'\otimes w' \rangle 
&= \langle \sigma(|v|, |w|)\, v \otimes w,\ v' \otimes w'\rangle_{\mathrm{ord}}\\
&= \langle v \otimes w,\ \sigma(|v'|, |w'|)\, v' \otimes w'\rangle_{\mathrm{ord}}; 
\end{align*}
see \eqref{eq:tensor_structure}. 
This shows that the dual (Hopf) superalgebra $\mathcal{H}^{\bar{*}}$ of a super-coalgebra
(or $\Bbbk$-finite projective Hopf superalgebra) $\mathcal{H}$  
coincides with the deformation ${}_{\sigma}(\mathcal{H}^*)$ of the one $\mathcal{H}^*$ 
treated in \cite{M2, M3, MS}. 

\subsection{Base extensions}\label{subsec:base_extension_Hopf}
Let $A \in \mathsf{SAlg}_{\Bbbk}$. Given $V \in \mathsf{SMod}_{\Bbbk}$, we let 
\[
V_A = A \otimes V \in A\text{-}\mathsf{SMod}
\]
denote that base extension to $A$. Given a pairing 
$\langle \ , \ \rangle : V \times W \to \Bbbk$
over $\Bbbk$, we let
\[
\langle \ , \ \rangle_A : V_A \times W_A \to A
\]
denote the base extension to $A$. This is a pairing over $A$. The base extension to $A$
of a Hopf superalgebra over $\Bbbk$ is a Hopf superalgebra over $A$. The base extension to
$A$ of a Hopf pairing over $\Bbbk$ is a Hopf pairing over $A$. 

\subsection{Lie superalgebras}\label{subsec:superLie}
A \emph{Lie superalgebra} (over $\Bbbk$) is an object $\mathfrak{g}$ given a morphism 
$[\ , \ ] : \mathfrak{g} \otimes \mathfrak{g} \to \mathfrak{g}$, both in $\mathsf{SMod}_{\Bbbk}$, 
such that 
\begin{itemize}
\item[(B1)] $[u,u]=0,\ u \in \mathfrak{g}_0$,
\item[(B2)] $[[v,v],v] =0,\ v \in \mathfrak{g}_1$,
\item[(B3)]
$[\ , \ ] \circ (\mathrm{id}_{\mathfrak{g} \otimes \mathfrak{g}} + c_{\mathfrak{g},\mathfrak{g}})= 0$,
\item[(B4)] 
$[[\ , \ ],\ ] \circ (\mathrm{id}_{\mathfrak{g}\otimes \mathfrak{g}\otimes \mathfrak{g}} + 
c_{\mathfrak{g}, \mathfrak{g}\otimes \mathfrak{g}} + c_{\mathfrak{g} \otimes \mathfrak{g}, \mathfrak{g}}) = 0$.
\end{itemize}
In the last two equations, $c_{V,W}:V\otimes W \to W \otimes V$ denotes the super-symmetry \eqref{eq:super-symmetry}. 
We call $[\ , \ ]$ the \emph{super-bracket} of the Lie superalgebra.

As is well known, (B1) ensures the equality (B3) 
restricted to $\mathfrak{g}_0 \otimes \mathfrak{g}_0$. 
Recall that a \emph{Lie algebra} is a $\Bbbk$-module given a bracket which satisfies
(B1) and the Jacobi identity, that is, (B4) in the purely even situation; it is, therefore, 
the same as a purely even Lie superalgebra. It follows
that if $\mathfrak{g}$ is a Lie superalgebra, then $\mathfrak{g}_0$ is a Lie algebra.

Let $\mathfrak{g}$ be a Lie superalgebra. 

A $2$-\emph{operation} \cite[Definition 2.2.1]{G} on $\mathfrak{g}$ is a map  
$(\ )^{\langle 2 \rangle} : \mathfrak{g}_1\to \mathfrak{g}_0$ such that
\begin{itemize}
\item[(B5)] $(\lambda v)^{\langle 2 \rangle} = \lambda^2v^{\langle 2 \rangle}$,
\item[(B6)] $(v+w)^{\langle 2 \rangle} = v^{\langle 2 \rangle} +[v,w]+ w^{\langle 2 \rangle}$,
\item[(B7)] $[v^{\langle 2 \rangle}, z] = [v,[v,z]]$,
\end{itemize}
where $\lambda \in \Bbbk$,\ $v, w \in \mathfrak{g}_1$,\ $z\in \mathfrak{g}$. 

Given $R \in \mathsf{Alg}_{\Bbbk}$, the base extension $\mathfrak{g}_R=R\otimes \mathfrak{g}$ 
is naturally a Lie superalgebra over $R$. 

\begin{lemma}[\text{\cite[Proposition A.3]{MS}}]\label{lem:base_extension_2-operation}
Assume that $\mathfrak{g}_1$ is $\Bbbk$-free.
Given a $2$-operation $(\ )^{\langle 2 \rangle} : \mathfrak{g}_1\to \mathfrak{g}_0$,
there uniquely exists a map 
\[
(\ )_R^{\langle 2 \rangle} : (\mathfrak{g}_1)_R  \to (\mathfrak{g}_0)_R
\]
such that
\[
\big(\sum_{i=1}^n c_i \otimes v_i \big)_R^{\langle 2 \rangle} = 
\sum_{i=1}^n c_i^2 \otimes v_i^{\langle 2 \rangle} + \sum_{i<j}c_ic_j\otimes [v_i, v_j]
\]
for every element $\sum_{i=1}^n c_i \otimes v_i\in R \otimes \mathfrak{g}_1$. 
This $(\ )_R^{\langle 2 \rangle}$ is a $2$-operation on the 
Lie superalgebra $\mathfrak{g}_R$ over $R$. 
\end{lemma}

Let $\mathfrak{g}$ be a Lie superalgebra equipped with a $2$-operation. The tensor algebra
$\mathbf{T}(\mathfrak{g})$ on $\mathfrak{g}$ uniquely turns into a 
Hopf superalgebra in which every even or odd element of $\mathfrak{g}$ is an even or odd primitive,
respectively. We let $\mathbf{U}(\mathfrak{g})$ denote the quotient Hopf superalgebra of  
$\mathbf{T}(\mathfrak{g})$ divided by the super-ideal
generated by the homogeneous primitives
\begin{equation}\label{eq:generators_of_super-ideal}
zw -(-1)^{|z||w|}wz-[z,w],\quad v^2 - v^{\langle 2 \rangle}, 
\end{equation}
where $z, w \in \mathfrak{g}$, and $v \in \mathfrak{g}_1$.
With the construction applied to $\mathfrak{g}_0$, 
obtained is the universal enveloping algebra $U(\mathfrak{g}_0)$, as is well known.
The inclusion $\mathfrak{g}_0 \to \mathfrak{g}$ induces a Hopf superalgebra map
$U(\mathfrak{g}_0) \to \mathbf{U}(\mathfrak{g})$, through which $\mathbf{U}(\mathfrak{g})$
turns into a left (and right) $U(\mathfrak{g}_0)$-module. 

Given an element $v$ of $\mathfrak{g}$, we will denote
its natural image in $\mathbf{U}(\mathfrak{g})$ by the same symbol $v$. 

\begin{prop}[\text{\cite[Corollary~A.6]{MS}}]\label{prop:PBW}
Let $\mathfrak{g}$ be as above. Assume
\begin{itemize}
\item[(C)] $\mathfrak{g}_0$ is $\Bbbk$-finite projective, and $\mathfrak{g}_1$ is $\Bbbk$-finite free.
\end{itemize}
Choose arbitrarily a $\Bbbk$-free basis $v_1,\dots, v_n$ of $\mathfrak{g}_1$. Then the left
$U(\mathfrak{g}_0)$-module $\mathbf{U}(\mathfrak{g})$ is free with the free basis
\begin{equation}\label{eq:basis_elements}
v_{i_1}v_{i_2}\dots v_{i_r}, 
\end{equation}
where $1 \le i_1 < i_2 < \dots < i_r \le n$,\ $r \ge 0$. 
\end{prop} 

It is known that if $\mathfrak{g}_0$ is $\Bbbk$-flat, then the canonical map $\mathfrak{g}_0\to U(\mathfrak{g}_0)$
is injective. Combined with the proposition above, it follows that under (C), the canonical map
$\mathfrak{g} \to \mathbf{U}(\mathfrak{g})$ is injective, and 
$\mathfrak{g}_1 \to \mathbf{U}(\mathfrak{g}) \leftarrow U(\mathfrak{g}_0)$ are $\Bbbk$-linearly split injections. 

\begin{rem}[\text{see \cite[Lemma A.2]{MS}}]\label{rem:2-torsion_free}
(1)\
Assume that $\Bbbk$ is $2$-torsion free in the
sense that $2 : \Bbbk \to \Bbbk$ is injective.
Let $\mathfrak{g}$ be a Lie superalgebra which satisfies 
\begin{itemize}
\item[(C$'$)] $\mathfrak{g}_0$ is $\Bbbk$-flat, and $\mathfrak{g}_1$ is $\Bbbk$-free.
\end{itemize} 
There exists a $2$-operation on $\mathfrak{g}$ if and only if for every $v \in \mathfrak{g}_1$,
the element $[v,v]$ is $2$-divisible in $\mathfrak{g}_0$, that is, it is the double of some
element of $\mathfrak{g}_0$; this last element is uniquely determined, and is 
denoted by $\frac{1}{2}[v,v]$.  
If this is the case, then 
\begin{equation}\label{eq:unique_2-operation}
v^{\langle 2 \rangle} := \frac{1}{2}[v,v], \quad v \in \mathfrak{g}_1
\end{equation}
defines a unique $2$-operation on $\mathfrak{g}$,
and the same result \cite[Proposition 3.4]{MS} as
Proposition \ref{prop:PBW} above is proved. 

(2)\
Assume that $\Bbbk$ is a field of characteristic $\ne 2$. 
Then the results above can apply:\ every Lie superalgebra has the unique $2$-operation defined
by \eqref{eq:unique_2-operation}, and 
one may not refer to such operations any more. 
\end{rem}

Let $\mathbf{G}$ be an affine supergroup, and set $\mathbf{O}=\mathcal{O}(\mathbf{G})$, 
the super-commutative Hopf superalgebra which represents $\mathbf{G}$. 
We let
\[
\mathbf{O}^+ = \operatorname{Ker}(\varepsilon : \mathbf{O}\to \Bbbk)
\]
denote the augmentation super-ideal of $\mathbf{O}$. The \emph{Lie superalgebra} $\operatorname{Lie}(\mathbf{G})$
of $\mathbf{G}$ is defined by
\[ 
\operatorname{Lie}(\mathbf{G})=(\mathbf{O}^+/(\mathbf{O}^+)^2)^*
\] 
as an object in $\mathsf{SMod}_{\Bbbk}$. We suppose
\[ 
\operatorname{Lie}(\mathbf{G})\subset \mathbf{O}^{\bar{*}}.
\] 
Indeed, $\operatorname{Lie}(\mathbf{G})$ is identified with the $\Bbbk$-super-submodule 
of $\mathbf{O}^{\bar{*}}$
which consists of the elements $z$ of $\mathbf{O}^{\bar{*}}$ such that
\[
z(hk) = z(h)\varepsilon(k) + \varepsilon(h)z(k),\quad h, k \in \mathbf{O}. 
\]

\begin{lemma}[\text{\cite[Proposition A.7]{MS}}]\label{lem:Lie(G)} 
$\operatorname{Lie}(\mathbf{G})$ is a Lie superalgebra under the super-bracket
\[
[z,w] := zw- (-1)^{|z||w|} wz,\quad z, w \in \operatorname{Lie}(\mathbf{G}),
\]
and a $2$-operation on $\operatorname{Lie}(\mathbf{G})$ is given by the square map
\[
(\ )^2 : \operatorname{Lie}(\mathbf{G})_1 \to \operatorname{Lie}(\mathbf{G})_0, \quad v \mapsto v^2.
\]
Here the products $zw$, $wz$ and the square $v^2$ are computed in $\mathbf{O}^{\bar{*}}$. 
\end{lemma}

We define $\operatorname{Lie}(\mathbf{G})$ to be the Lie superalgebra equipped with the $2$-operation, as above.
One sees that $\operatorname{Lie}$ then gives a functor from the category of affine supergroups
to the category of Lie superalgebras equipped with $2$-operation. A morphism of the latter category 
is a morphism in $\mathsf{SMod}_{\Bbbk}$ which preserves the super-bracket and the $2$-operation. 

\begin{rem}\label{rem:deform_Lie}
Let $\mathfrak{g}$ be a Lie superalgebra. Note from Section \ref{subsec:compare} that the deformation 
${}_{\sigma}\mathfrak{g}$ by $\sigma$ is the object $\mathfrak{g}$ in $\mathsf{SMod}_{\Bbbk}$ 
which is given the super-bracket
\[
{}_{\sigma}[z, w] := (-1)^{|z||w|}[z, w],\quad z, w \in \mathfrak{g}
\]
deformed from the original super-bracket $[z, w]$. If $\mathfrak{g}$ is equipped with a $2$-operation, we suppose
that ${}_{\sigma}\mathfrak{g}$ is equipped with the deformed $2$-operation
\[
v^{\, {}_{\sigma}\hspace{-0.5mm}\langle 2 \rangle} := - v^{\langle 2 \rangle},\quad v \in\mathfrak{g}_1. 
\]
This indeed defines a $2$-operation on ${}_{\sigma}\mathfrak{g}$, as is easily seen. 

Let $\mathbf{G}$ be an affine supergroup. As is seen from the last paragraph of Section \ref{subsec:compare},
the definition of $\operatorname{Lie}(\mathbf{G})$ above 
is different from the one in
\cite[Appendix]{MS}. In fact, the two $\operatorname{Lie}(\mathbf{G})$
are the deformations of each other by $\sigma$. 
\end{rem}

\subsection{$G$-supermodules}\label{subsec:G-supermodule}
Let $G$ be an affine group, and set $O = \mathcal{O}(G)$.
  
Let $V \in \mathsf{SMod}_{\Bbbk}$. 
A \emph{right} $G$-\emph{supermodule structure} on $V$ is an anti-morphism from $G$ to the
group functor $\mathsf{GL}(V)$ which assigns to each $R \in \mathsf{Alg}_{\Bbbk}$, the group 
$\mathrm{GL}_R(V_R)$ of all $R$-super-linear automorphisms of $V_R$. Such a structure, say $\alpha$, 
arises uniquely from a left $O$-super-comodule structure
\begin{equation}\label{eq:rho}
\rho_{\alpha} : V \to O \otimes V,\quad \rho_{\alpha}(v) = v_{(-1)}\otimes v_{(0)} 
\end{equation}
so that 
\[ 
\alpha_R(g)(1 \otimes v)= g(v_{(-1)}) \otimes v_{(0)},\quad v \in V,\ g \in G(R), 
\]
where $R \in \mathsf{Alg}_{\Bbbk}$. Note that $\rho_{\alpha}(v)=\alpha_O(\mathrm{id}_O)(1 \otimes v)$. 
We will write $v^g$ for $\alpha_R(g)(1 \otimes v)$, and $u^g$ for $\alpha_R(g)(u)$, $u \in V_R$. 

Similarly, for left $G$-supermodule structures we will write as ${}^gw$. 

Let $W \in \mathsf{SMod}_{\Bbbk}$, and suppose that it is $\Bbbk$-finite projective. 
A left $G$-supermodule structure on $W$ is \emph{transposed} to $W^*$ so that
\begin{equation}\label{eq:transpose}
v^{g}(w):= v({}^g w),\quad v \in W^*,\ w \in W,\ g \in G(R). 
\end{equation}
This defines a right $G$-supermodule structure on $W^*$. Here we understand that the $v$ in the
equation represents the element $1 \otimes v \in (W^*)_R$ or its image through
the canonical isomorphism $(W^*)_R\simeq \operatorname{Hom}_R(W_R, R)$. 

The right (or left) $G$-supermodules form a symmetric tensor category with respect to the tensor
product $\otimes$, the unit object $\Bbbk$ and the super-symmetry. 

\section{Construction of affine supergroups based on Gavarini's idea}\label{sec:construction}

In this section the base commutative ring $\Bbbk$ remains arbitrary.

\subsection{The group $\mathbf{\Sigma}(A)$}\label{subsec:PhiA}
Let $\mathfrak{g}$ be a Lie superalgebra which satisfies (C); see Proposition \ref{prop:PBW}. 
Suppose that it is equipped with a $2$-operation. 

Let $A\in \mathsf{SAlg}_{\Bbbk}$. 
We have the group $\mathsf{Gpl}(\mathbf{U}(\mathfrak{g})_A)$ of all even grouplikes in
the Hopf superalgebra $\mathbf{U}(\mathfrak{g})_A = A \otimes \mathbf{U}(\mathfrak{g})$ over $A$. 
As is seen from the paragraph following Proposition \ref{prop:PBW}, the canonical maps 
\[ A_0 \otimes \mathfrak{g}_0 \to A \otimes U(\mathfrak{g}_0) \to 
A \otimes \mathbf{U}(\mathfrak{g}) \leftarrow A \otimes \mathfrak{g}_1 \]
are all injections, which we will regard as inclusions. We define even elements 
$e(a,v)$, $f(\epsilon, x)$ of $A \otimes \mathbf{U}(\mathfrak{g})$ by
\begin{equation}\label{eq:ef}
e(a,v)=1 \otimes 1 + a\otimes v,\quad f(\epsilon, x)= 1\otimes 1+ \epsilon \otimes x,
\end{equation}
where  $a \in A_1$, $v \in \mathfrak{g}_1$, $x \in \mathfrak{g}_0$, and $\epsilon \in A_0$ with
$\epsilon^2=0$. 
Note that $e(\lambda a,v)=e(a,\lambda v)$, $f(\lambda\epsilon, x)=f(\epsilon, \lambda x)$ for $\lambda \in \Bbbk$. 

\begin{lemma}\label{lem:e(a,v)}
The elements $e(a,v)$,\ $f(\epsilon, x)$ are contained 
in $\mathsf{Gpl}(\mathbf{U}(\mathfrak{g})_A)$, and we have
\[ e(a, v)^{-1}= e(-a,v),\quad f(\epsilon, x)^{-1} = f(-\epsilon, x),\quad 
e(0, v) = 1 = f(0, x). \]
\end{lemma}
\begin{proof}
This follows since $a \otimes v$ and $\epsilon \otimes x$ are such even primitives $z$ 
that are tensor-square zero, $z\otimes_A z=0$. 
\end{proof}

\begin{lemma}\label{lem:relations}
Let $a, b \in A_1$,\ $u, v \in \mathfrak{g}_1$,\ $x, y \in \mathfrak{g}_0$,\ and $\epsilon, \eta \in A_0$
with $\epsilon^2=\eta^2=0$. Then  
the following relations hold in $\mathsf{Gpl}(\mathbf{U}(\mathfrak{g})_A)$. 
\begin{itemize}
\item[(i)] $e(a,u)\, e(b,v) = f(-ab,[u,v])\, e(b,v)\, e(a,u)$
\item[(ii)] $e(a, v)\, e(b, v) = f(-ab,v^{\langle 2 \rangle})\, e(a+b, v)$
\item[(iii)] $e(a, v)\, f(\epsilon, x)= f(\epsilon, x)\, e(a, v)\, e(\epsilon a, [v,x])$ 
\item[(iv)] $f(\epsilon, x)\, f(\eta, y) = f(\eta, y)\, f(\epsilon, x)\, f(\epsilon \eta, [x,y])$
\end{itemize}
\end{lemma}
\begin{proof}
These follow by direct computation. 
\end{proof}

In particular, $e(a,u)$ and $e(b,v)$ (resp., $e(a, v)$ and $f(\epsilon, x)$;\ resp.,  
$f(\epsilon, x)$ and $f(\eta, y)$) commute with each other if $ab=0$ or $[u,v]=0$ 
(resp., if $\epsilon a=0$ or $[v,x]=0$;\ resp., if $\epsilon \eta=0$ or $[x,y]=0$). 

Let 
$\mathbf{\Sigma}(A)$
denote the subgroups of $\mathsf{Gpl}(\mathbf{U}(\mathfrak{g})_A)$ generated by all the elements
$e(a,v)$, $f(\epsilon, x)$ defined by \eqref{eq:ef}. Let
$F(A_0)$
denote the subgroup of $\mathbf{\Sigma}(A)$ generated by all $f(\epsilon, x)$.  

\begin{prop}\label{prop:PhiA}
We have the following.
\begin{itemize}
\item[(1)] $F(A_0)= \mathbf{\Sigma}(A)\cap U(\mathfrak{g}_0)_{A_0}$. 
\item[(2)] Choose arbitrarily a $\Bbbk$-free basis $v_1,\dots, v_n$ of $\mathfrak{g}_1$. Then
every element of $\mathbf{\Sigma}(A)$ is uniquely expressed in the form
\begin{equation}\label{eq:expression}
f\, e(a_1,v_1)\, e(a_2,v_2)\dots e(a_n,v_n), 
\end{equation}
where $f \in F(A_0)$, and $a_i\in A_1$, $1\le i \le n$. 
\end{itemize}
\end{prop}
\begin{proof}
First, we prove the uniqueness of expression in Part 2. 
Note that $F(A_0) \subset A \otimes U(\mathfrak{g}_0)$. By Proposition \ref{prop:PBW} 
$\mathbf{U}(\mathfrak{g})_A$ has the elements given by \eqref{eq:basis_elements} as left
$A \otimes U(\mathfrak{g}_0)$-free basis. Suppose that one element has two expressions, 
$f\, e(a_1,v_1) \dots e(a_n,v_n)$,\ 
$f'\, e(a'_1,v_1) \dots e(a'_n,v_n)$. Then the comparison of
coefficients of the free basis elements $1\otimes 1$, $1\otimes v_i$ shows that $f=f'$, $fa_i=f'a'_i$, and so
$a_i=a'_i$, which proves the uniqueness. 

The argument also shows that the element 
$f\, e(a_1,v_1) \dots e(a_n,v_n)$ 
is in $A \otimes U(\mathfrak{g}_0)$ if and only if $a_i=0$, $1\le i\le n$. 
Therefore, once the possibility of expression in Part 2 is shown, Part 1 follows. 
\end{proof}

The proof of the last mentioned possibility which was given in an earlier version of this paper was wrong. 
Alexandr Zubkov kindly pointed out this, 
showing a correct proof which is reproduced essentially as follows. 

\begin{lemma}[A.~Zubkov]\label{lem:Zubkov}
In the situation of Proposition \ref{prop:PhiA}, choose arbitrarily
a super-ideal $\mathfrak{a}$ of $A$. For each integer $k \ge 0$, let $\mathbf{\Sigma}_k$ (resp., $F_k$) denote
the subgroup of $\mathbf{\Sigma}(A)$ (resp., of $F(A_0)$) which is generated by the elements  
$e(a, v)$ and $f(\epsilon, x)$ (resp., the elements $f(\epsilon, x)$), where $a \in \mathfrak{a}^k\cap A_1$,
$v \in \mathfrak{g}_1$, $x \in \mathfrak{g}_0$,
and $\epsilon \in \mathfrak{a}^k\cap A_0$ with $\epsilon^2=0$.
\begin{itemize}
\item[(1)] We have \[ [\mathbf{\Sigma}_k,\mathbf{\Sigma}_\ell]\subset \mathbf{\Sigma}_{k+\ell},\quad
[F_k,F_\ell]\subset F_{k+\ell}\ \, \text{for all}\ \,  k, \ell \ge 0.\] 
\item[(2)] Each $\mathbf{\Sigma}_k\subset \mathbf{\Sigma}(A)$ (resp., $F_k\subset F(A_0)$) is a normal subgroup such that
$\mathbf{\Sigma}_k \supset F_k$, $k \ge 0$, and
\[ \mathbf{\Sigma}(A)=\mathbf{\Sigma}_0 \supset \mathbf{\Sigma}_1 \supset \mathbf{\Sigma}_2 \supset \dots;
\quad F(A_0)=F_0 \supset F_1 \supset F_2 \supset \dots \]
\item[(3)] Let $v \in \mathfrak{g}_1$, 
$a \in \mathfrak{a}^k\cap A_1$ and 
$b \in \mathfrak{a}^{\ell}\cap A_1$, where $k$, $\ell\ge 0$ are integers. Then we have
\[e(a,v)\, e(b,v)\equiv e(a+b,v)\ \mathrm{mod}\, F_{k+\ell}.\] 
\item[(4)] Choose arbitrarily a $\Bbbk$-free basis $v_1, . . . , v_n$ of $\mathfrak{g}_1$. 
Fix an integer $k >1 $. Then every element of $\mathbf{\Sigma}_1$ is congruent modulo $\mathbf{\Sigma}_{k}$ to 
a product 
\[ f \, e(a_1, v_1)\, e(a_2, v_2) \dots e(a_n, v_n), \]
where $f \in F_1$, and $a_i\in \mathfrak{a} \cap A_1$,\ $1 \le i \le n$.
\end{itemize}
\end{lemma}
\begin{proof}
(1), (3) These follow from Lemmas \ref{lem:e(a,v)}--\ref{lem:relations}. 

(2) This follows from (1). 

(4) We prove by induction on $k$.
If $v = \sum_{i=1}^n \lambda_i v_i$ with $\lambda_i \in \Bbbk$, then 
\[ e(a, v) = e(\lambda_1a, v_1)\, e(\lambda_2a, v_2)\dots e(\lambda_na, v_n). \]
Therefore, $\mathbf{\Sigma}_1$ is generated by $F_1$ and all the elements
$e(a, v_i)$,\ $1\le i \le n$, where $a_i\in \mathfrak{a}\cap A_1$.
Since $\mathbf{\Sigma}_1/\mathbf{\Sigma}_2$ is abelian by (1) the desired result for $k=2$ follows
by (3).

An analogous result on the group $\mathbf{\Sigma}_k/\mathbf{\Sigma}_{k+1}$ which is seen to be abelian, combined
with the induction hypothesis, shows that every element 
of $\mathbf{\Sigma}_1$ is congruent modulo $\mathbf{\Sigma}_{k+1}$ to 
a product 
\[ f \, e(a_1, v_1) \dots e(a_n, v_n)\, f' \, e(a'_1, v_1) \dots e(a'_n, v_n), \]
where $f \in F_1$, $f'\in F_k$, 
$a_i\in \mathfrak{a} \cap A_1$ and $a'_i\in \mathfrak{a}^k \cap A_1$, $1 \le i \le n$.
This is congruent to
\[ ff'\, e(a_1,v_1) \, e(a'_1, v_1) \dots e(a_n,v_n) \, e(a'_n, v_n) \]
since $\mathbf{\Sigma}_k/\mathbf{\Sigma}_{k+1}$ is central in $\mathbf{\Sigma}_1/\mathbf{\Sigma}_{k+1}$ by (1). 
The desired
result for $k+1$ follows by (3).
\end{proof}

\begin{proof}[Proof of Proposition \ref{prop:PhiA} (Continued)]
To express as a desired product, an element, say $h$, which is the product of any order of elements
\[ e(\tau_i, v_i),\ 1\le i \le n;\ f(\epsilon_j,x_j),\ 1\le j \le m,\]
where $v_i \in \mathfrak{g}_1$, $x_j \in \mathfrak{g}_0$, we may suppose that
$\tau_i$ and $\epsilon_j$ are variables, or more precisely, we may suppose
\[ A= \Bbbk[\epsilon_1,\dots,\epsilon_m]/(\epsilon_1^2,\dots, \epsilon_m^2)\otimes \wedge(\tau_1,\dots,\tau_n) \]
in which $\tau_i$, $1\le i \le n$, are odd variables; for arbitrary $\tau'_i$ and $\epsilon'_j$ in $B$, say,
one has only to specialize the obtained expression $h =f\, e(a_1,v_1)\dots e(a_n,v_n)$ via
$\mathbf{\Sigma}(A) \to \mathbf{\Sigma}(B)$ induced by $A\to B,\ \tau_i\mapsto \tau'_i,\ \epsilon_j \mapsto \epsilon'_j$. 
Lemma \ref{lem:Zubkov}, applied to this $A$ and the super-deal
$\mathfrak{a}=(\epsilon_1,\dots, \epsilon_m,\tau_1,\dots,\tau_n)$, gives the desired expression,
since $\mathfrak{a}$ is nilpotent, and so $\mathbf{\Sigma}_k$ is trivial for $k\gg 0$.
\end{proof}


\subsection{The group $\mathbf{\Gamma}(A)$}\label{subsec:GammaA}
Retain the situation as above.
 
Let $G$ be an affine group. The right adjoint action $G \times G \to G$, $(h,g) \mapsto g^{-1}hg$
is dualized to the left $G$-module structure on $\mathcal{O}(G)$ defined by
\begin{equation}\label{eq:gc}
{}^gc = g^{-1}(c_{(1)})\, c_{(2)}\, g(c_{(3)}), \quad g\in G(R), \ c \in \mathcal{O}(G), 
\end{equation}
where $R \in \mathsf{Alg}_{\Bbbk}$. This makes $\mathcal{O}(G)$ into a Hopf-algebra object in the
symmetric tensor category $G\text{-}\mathsf{Mod}$ of left $G$-modules.  

Recall that $\mathfrak{g}$ is a Lie superalgebra equipped with a $2$-operation, and it satisfies (C). 
Let $\mathsf{Aut}_{Lie}(\mathfrak{g})$ denote the subgroup functor of $\mathsf{GL}(\mathfrak{g})$
(see Section \ref{subsec:G-supermodule}) 
that assigns to each $R \in \mathsf{Alg}_{\Bbbk}$, the group 
$\operatorname{Aut}_{R\text{-}Lie}(\mathfrak{g}_R)$ 
of all $R$-Lie-superalgebra automorphisms preserving $(\ )_R^{\langle 2 \rangle}$; 
see Lemma \ref{lem:base_extension_2-operation}. 

We are going to work in a more general situation than will be needed to 
discuss a category equivalence in Section \ref{subsec:reprove}; see 
Remark \ref{rem:super_Lie_group} for the reason. 

Suppose that we are given a pairing and an anti-morphism, 
\begin{equation}\label{eq:pairing_alpha}
\langle \ , \ \rangle : \mathfrak{g}_0 \times \mathcal{O}(G) \to \Bbbk,\quad
\alpha : G \to \mathsf{Aut}_{Lie}(\mathfrak{g}).
\end{equation}
As in \eqref{eq:rho}, let us write as
$\rho_{\alpha}(z) = z_{(-1)} \otimes z_{(0)}$,\ $z \in \mathfrak{g}$.  
We assume that
\begin{itemize}
\item[(D1)] $[z,x]= \langle x,\ z_{(-1)}\rangle \, z_{(0)}$, 
\item[(D2)] $\langle x,\ cd\rangle = \langle x,\ c \rangle \, \varepsilon(d)
+ \varepsilon(c)\, \langle x,\ d \rangle$, and
\item[(D3)] $\langle x^g,\ c \rangle_R = \langle x,\ {}^g c \rangle_R$,
\end{itemize}
where $x \in \mathfrak{g}_0$, $z \in \mathfrak{g}$, $c, d \in \mathcal{O}(G)$ and 
$g \in G(R)$, $R \in \mathsf{Alg}_{\Bbbk}$.

By (D2) we have the map
\begin{equation}\label{eq:i_Lie_map}
\mathfrak{g}_0 \to \operatorname{Lie}(G)\, (\subset \mathcal{O}(G)^*),\quad x\mapsto \langle x,\ - \rangle.  
\end{equation}
This is a Lie algebra map, since we see from (D1) for even $z$ and (D3) that
\begin{align*}
\langle [x,y],\ c \rangle 
&= \langle x, c_{(2)}\rangle\, \langle y,\ \mathcal{S}(c_{(1)})c_{(3)}\rangle \\
&= \langle x\otimes y,\ \Delta(c)\rangle - \langle y \otimes x,\ \Delta(c)\rangle, 
\end{align*}
where $x, y \in \mathfrak{g}_0$,\ $c \in \mathcal{O}(G)$. 
Therefore, it uniquely extends to an algebra map $U(\mathfrak{g}_0) \to \mathcal{O}(G)^*$, 
with which is associated the Hopf pairing 
\begin{equation}\label{eq:Hopf_pair}
\langle\ , \ \rangle : U(\mathfrak{g}_0) \times \mathcal{O}(G) \to \Bbbk 
\end{equation}
that uniquely extends the given pairing.

Recall $A \in \mathsf{SAlg}_{\Bbbk}$. By Lemma \ref{lem:group_map} the base extension to $A_0$ of the last Hopf
pairing gives rise to the group map
\[
\mathsf{Gpl}(U(\mathfrak{g}_0)_{A_0}) \to \mathsf{Alg}_{\Bbbk}(\mathcal{O}(G), A_0)=G(A_0),\quad g \mapsto 
\langle g, -\rangle_{A_0},
\]
whose restriction to $F(A_0)$ we denote by
\[ i_{A_0}=i : F(A_0) \to G(A_0). \]

\begin{lemma}\label{lem:extension_to_Hopf_automorphism}
Let $R \in \mathsf{Alg}_{\Bbbk}$ and $g \in G(R)$. 
Then $\alpha_R(g) \in \operatorname{Aut}_{R\text{-}Lie}(\mathfrak{g}_R)$  
uniquely extends to an automorphism
of the Hopf superalgebra $\mathbf{U}(\mathfrak{g})_R$ over $R$. 
\end{lemma}
\begin{proof}
One sees that $\alpha_R(g)$ uniquely extends an automorphism of the $R$-Hopf 
superalgebra $\mathbf{T}(\mathfrak{g})_R$. 
It is easy to see that the automorphism stabilizes the super-ideal of $\mathbf{T}(\mathfrak{g})_R$
generated by the elements
$zw -(-1)^{|z||w|}wz-[z,w]$ in \eqref{eq:generators_of_super-ideal}. To see that it stabilizes 
the super-ideal generated by all elements in \eqref{eq:generators_of_super-ideal}, let $v \in \mathfrak{g}_1$,
and suppose $v^g= \sum_i c_i \otimes v_i\in R \otimes \mathfrak{g}_1$. 
Then the desired result will follow if one compares the following two:
\begin{align*}
&(v^{\langle 2 \rangle})^g = (v^g)^{\langle 2 \rangle}_R 
= \sum_i c_i^2 \otimes v_i^{\langle 2 \rangle} + \sum_{i<j}c_i c_j\otimes [v_i, v_j],\\
&(v^2)^g = (v^g)^2 
= \sum_i c_i^2 \otimes v_i^2 + \sum_{i<j}c_i c_j \otimes (v_iv_j+v_j v_i). 
\end{align*}
\end{proof}
 
The assignment of the above extended automorphism to $g \in G(R)$
gives rise to an anti-morphism 
from $G$ to the automorphism group functor of $\mathbf{U}(\mathfrak{g})$, 
which we denote again by
\[ 
\alpha : G \to \mathsf{Aut}_{Hopf}(\mathbf{U}(\mathfrak{g})). 
\] 
Given $g \in G(A_0)$, the base extension $(\alpha_{A_0}(g))_A$ 
of $\alpha_{A_0}(g) \in \operatorname{Aut}_{A_0\text{-}Hopf} \ (\mathbf{U}(\mathfrak{g})_{A_0})$ along
$A_0 \to A$ is an automorphism of the Hopf superalgebra $\mathbf{U}(\mathfrak{g})_A$ over $A$. 
As before, we will write $u^g$ for $(\alpha_{A_0}(g))_A(u)$, where $u \in \mathbf{U}(\mathfrak{g})_A$,\
$g \in G(A_0)$. 
Since the action 
stabilizes $\mathbf{\Sigma}(A)$, as is seen from the next lemma, 
it follows that $g \mapsto (\alpha_{A_0}(g))_A|_{\Sigma(A)}$ defines a anti-group map 
from $G(A_0)$ to the automorphism group of the group $\mathbf{\Sigma}(A)$, 
which we denote by 
\[ 
\alpha_A : G(A_0) \to \operatorname{Aut}(\mathbf{\Sigma}(A)). 
\]

\begin{lemma}\label{lem:egfg} 
Let $g \in G(A_0)$. Let $e(a, v)$ and $f(\epsilon, x)$ be as before. Suppose  
\[
\rho_{\alpha}(v) = \sum_{i=1}^n c_i \otimes v_i \in \mathcal{O}(G)\otimes \mathfrak{g}_1,\quad 
\rho_{\alpha}(x) = \sum_{j=1}^m d_j \otimes x_j \in \mathcal{O}(G)\otimes \mathfrak{g}_0. 
\]  
Then we have
\begin{itemize}
\item[(1)] $e(a, v)^g= e(a g(c_1),v_1)\, e(a g(c_2), v_2) \dots e(a g(c_n), v_n)$,
\item[(2)] $f(\epsilon, x)^g = f(\epsilon g(d_1), x_1)\, f(\epsilon g(d_2), x_2)
\dots f(\epsilon g(d_m), x_m).$
\end{itemize}
\end{lemma}

This is easy to see. We remark that $F(A_0)$ is $G(A_0)$-stable by Part 2. 

\begin{prop}\label{prop:quintuple}
The quintuple
\[ 
(\mathbf{\Sigma}(A),\ F(A_0),\ G(A_0),\ i_{A_0},\ \alpha_A)
\]
satisfies Conditions (A1)--(A3) given in Section \ref{sec:base_extension}. 
\end{prop}
\begin{proof}
Since the last remark shows that the first half of (A3) is satisfied, 
it remains to verify (A2) and the second half of (A3). 

Choose $g \in G(A_0)$, and let $f = f(\epsilon, x)$. 
Note   
\[ i(f)(c) = \varepsilon(c)1 + \epsilon \langle x,\ c \rangle,\quad c \in \mathcal{O}(G). \]
Then by using (D3) we see
\begin{align*}
i(f^g)(c) &= \varepsilon(c)1 + \epsilon\, \langle x^g,\ c\rangle_{A_0}  
= \varepsilon(c)1 + \epsilon\, \langle x,\ {}^g c \rangle_{A_0}\\ 
&=\varepsilon(c)1 + \epsilon\, g^{-1}(c_{(1)})\, \langle x,\ c_{(2)}\rangle\, g(c_{(3)})\\
&= (g^{-1}i(f)g)(c),
\end{align*}
which verifies the second half of (A3). By using (D1) we see
\begin{align*}
e(a, v)^{i(f)}&= 1 \otimes 1 + a \, i(f)(v_{(-1)})\otimes v_{(0)}\\
&= 1 \otimes 1 +a \otimes v + \epsilon a\otimes \langle x,\ v_{(-1)}\rangle v_{(0)}\\
&= 1 \otimes 1 + a \otimes v + \epsilon a \otimes [v,x]\\ 
&= e(a, v)\, e(\epsilon a, [v,x]),
\end{align*}
and similarly,
\[
f(\eta, y)^{i(f)} = f(\eta, y)\, f(\epsilon \eta, [y,x]).
\]
These, combined with (iii)--(iv) of Lemma \ref{lem:relations}, verify (A2). 
\end{proof}

\begin{definition}\label{def:Gamma(A)}
$\mathbf{\Gamma}(A)$ denotes the base extension of $\mathbf{\Sigma}(A)$ along 
$i_{A_0} : F(A_0) \to (G(A_0), \alpha_A)$; see Definition \ref{def:base_extension}. 
\end{definition}

In $\mathbf{\Gamma}(A)$, the natural images of $e(a,v)$ and of elements $g \in G(A_0)$ 
will be denoted by the same symbols. 

\begin{prop}\label{prop:basis}
Choose arbitrarily a $\Bbbk$-free basis $v_1,\dots, v_n$ of $\mathfrak{g}_1$. 
Then every element of $\mathbf{\Gamma}(A)$ is uniquely expressed in the form 
\[ g\, e(a_1,v_1)\, e(a_2,v_2)\dots e(a_n,v_n), \]
where $g \in G(A_0)$,\ $a_i\in A_1$, $1\le i \le n$. 
\end{prop}
\begin{proof}
This follows from Proposition \ref{prop:PhiA} (2) and the proof of Lemma \ref{lem:base_extension} (2). 
\end{proof}

Gavarini's original construction starts with constructing by generators and relation
the group which shall be the functor points of the desired affine supergroup;
see the group $G_{\mathcal{P}}(A)$ defined by \cite[Definition 4.3.2]{G}.
Let us prove 
that the group, which is essentially the same as $\mathbf{\Gamma'}(A)$ below, 
is isomorphic to our $\mathbf{\Gamma}(A)$, though
the result will not be used in the subsequent argument. 

\begin{lemma}\label{lem:free_product}
Choose arbitrarily a $\Bbbk$-free basis $v_1,\dots, v_n$ of $\mathfrak{g}_1$.
Let $E(A_1)$ denote the free group on the set of the symbols
\[ e_j(a),\ 1 \le j \le n,\ a \in A_1,  \]
and let $\mathbf{\Gamma'}(A)$ denote the quotient group of the free product $G(A_0) * E(A_1)$ 
divided by the relations
\begin{itemize}
\item[(i)] $e_j(a)\, e_k(b) = i(f(-ab,[v_j,v_k]))\, e_k(b)\, e_j(a)$,\; $k < j$, 
\item[(ii)] $e_j(a)\, e_j(b) = i(f(-ab,v_j^{\langle 2 \rangle}))\, e_j(a+b)$,
\item[(iii)] $e_j(a)\, g = g\, e_1(a g(c_{j1})) \dots e_n(a g(c_{jn}))$, where $g \in G(A_0)$,\ 
and we suppose $\rho_{\alpha}(v_j)=\sum_{k=1}^n c_{jk} \otimes v_k$. 
\end{itemize} 
Then 
\[ e_j(a) \mapsto e(a, v_j),\quad 1\le j \le n,\ a \in A_1 \]
gives an isomorphism $\mathbf{\Gamma'}(A) \overset{\simeq}{\longrightarrow} \mathbf{\Gamma}(A)$
which is identical on $G(A_0)$.
\end{lemma} 
\begin{proof}
It is easy to see that the assignment above gives an epimorphism. By Proposition \ref{prop:basis}, 
$g\, e(a_1,v_1) \dots e(a_n,v_n)\mapsto g\, e_1(a_1) \dots e_n(a_n)$ well defines a section. 
This section is surjective, since one sees just as proving Proposition \ref{prop:PhiA}
that every element of $\mathbf{\Gamma'}(A)$ is expressed in the form $g \, e_1(a_1) \dots e_n(a_n)$, 
where $g \in G(A_0)$,\ $a_j \in A_1$, $1 \le j \le n$. The surjectivity of the section 
proves that the epimorphism is an isomorphism. 
\end{proof}

\subsection{The affine supergroup $\mathbf{\Gamma}$}\label{subsec:supergroup_Gamma}
Retain the situation as above. One sees easily that
\[ A \mapsto \mathbf{\Gamma}(A) \]
defines a group functor $\mathbf{\Gamma}$ on $\mathsf{SAlg}_{\Bbbk}$. Moreover, we see:

\begin{prop}\label{prop:Gamma_affine}
This $\mathbf{\Gamma}$ is an affine supergroup, represented by the super-commutative superalgebra
\begin{equation}\label{eq:OGamma}
\mathbf{O} := \mathcal{O}(G) \otimes \wedge(\mathfrak{g}_1^*). 
\end{equation}
\end{prop}
\begin{proof}
Choose a $\Bbbk$-free basis $v_1,\dots, v_n$ of $\mathfrak{g}_1$, as above. 
Let $w_1,\dots, w_n$ denote the dual basis of $\mathfrak{g}_1^*$. Proposition \ref{prop:basis} gives the bijection
\begin{equation}\label{eq:bijection_to_Gamma}
G(A) \times A_1^n \overset{\simeq}{\longrightarrow} \mathbf{\Gamma}(A),\ \; (g,a_1,\dots, a_n) \mapsto 
g\, e(a_1,v_1) \dots e(a_n,v_n),
\end{equation}
which is seen to be natural in $A$. To an element $(g,a_1,\dots, a_n) \in G(A) \times A_1^n$, assign the
superalgebra map $\phi : \mathbf{O} \to A$ defined by
\[ \phi(c) =g(c),\; c \in \mathcal{O}(G),\quad \phi(w_i)=a_i, \; 1\le i \le n. \]
This indeed defines such a map since every element in $A_1$ is required to be square-zero. 
The assignment above is in fact a bijection
\begin{equation}\label{eq:bijection_to_SAlg}
G(A) \times A_1^n\overset{\simeq}{\longrightarrow} \mathsf{SAlg}_{\Bbbk}(\mathbf{O}, A) 
\end{equation}
which is natural in $A$. This proves the proposition. 
\end{proof}
\begin{rem}\label{rem:O}
Note that $G$, regarded as $A \mapsto G(A_0)$, is a subgroup functor of $\mathbf{\Gamma}$. 
Let $\mathbf{G}_{\text{a}}^{-n}$ denote
the functor which assigns to $A\in \mathsf{SAlg}_{\Bbbk}$ the additive group $A_1^n$.
Note that this $\mathbf{G}_{\text{a}}^{-n}$ is represented by 
$\wedge(\mathfrak{g}_1^*)$. 
One sees that the bijection \eqref{eq:bijection_to_Gamma}
gives rise to a left $G$-equivariant isomorphism 
$G \times \mathbf{G}_{\text{a}}^{-n} \overset{\simeq}{\longrightarrow} \mathbf{\Gamma}$ of functors
which preserves the identity element.
\end{rem}

The superalgebra $\mathbf{O}$ has a unique Hopf-superalgebra structure that makes the composite
$\mathbf{\Gamma}(A)\overset{\simeq}{\longrightarrow}\mathsf{SAlg}_{\Bbbk}(\mathbf{O}, A)$ of
the bijections \eqref{eq:bijection_to_Gamma} and \eqref{eq:bijection_to_SAlg} into 
an isomorphism of group functors. 
In particular, the counit is the tensor product
\[ \varepsilon \otimes \varepsilon : \mathcal{O}(G) \otimes \wedge(\mathfrak{g}_1^*) \to \Bbbk \]
of the counits of the Hopf superalgebras $\mathcal{O}(G)$ and $\wedge(\mathfrak{g}_1^*)$, as is 
seen from Remark \ref{rem:O}. It follows that
\[ \mathbf{O}^+/(\mathbf{O}^+)^2 = \mathcal{O}(G)^+/(\mathcal{O}(G)^+)^2\oplus \, \mathfrak{g}_1^*, \]
which is dualized to the identification
\[ \operatorname{Lie}(\mathbf{\Gamma})= \operatorname{Lie}(G) \oplus \mathfrak{g}_1 \]
of $\Bbbk$-supermodules. 

Let  
$i' : \mathfrak{g}_0 \to \operatorname{Lie}(G)$ 
denote the Lie algebra map given by \eqref{eq:i_Lie_map}.
Let $\operatorname{Der}(\mathfrak{g})$ denote the Lie algebra of $\Bbbk$-super-linear 
derivations on $\mathfrak{g}$.
The morphism $\alpha$ given in \eqref{eq:pairing_alpha} induces the anti-Lie algebra map 
\[ 
\alpha' : \operatorname{Lie}(G)\to \operatorname{Der}(\mathfrak{g}),\quad 
\alpha'(x)(z)= x(z_{(-1)})\, z_{(0)},
\]
where $x \in \operatorname{Lie}(G),\ z \in \mathfrak{g}$. 
We remark that by (D1), the composite
$\alpha' \circ i' : \mathfrak{g}_0\to \operatorname{Der}(\mathfrak{g})$ coincides with the
right adjoint representation. 

\begin{prop}\label{prop:bracket}
We have the following.
\begin{itemize}
\item[(1)]
The super-bracket on $\operatorname{Lie}(\mathbf{\Gamma})= \operatorname{Lie}(G) \oplus \mathfrak{g}_1$
is given by
\[
[(x,u), (y,v)] = \big([x,y]+i'([u,v]), \; \alpha'(y)(u) - \alpha'(x)(v) \big), 
\]
where $x, y \in \operatorname{Lie}(G)$,\ $u, v \in \mathfrak{g}_1$.
\item[(2)] $i' \oplus \operatorname{id} : \mathfrak{g} = \mathfrak{g}_0 \oplus \mathfrak{g}_1
\to \operatorname{Lie}(G) \oplus \mathfrak{g}_1= \operatorname{Lie}(\mathbf{\Gamma})$ is a Lie superalgebra
map which preserves the $2$-operation.
\end{itemize} 
\end{prop}
\begin{proof}
(1)\  
We see from Remark \ref{rem:O} that
$\mathcal{O}(G)$ is a quotient Hopf superalgebra of $\mathbf{O}$
through $\operatorname{id} \otimes \varepsilon : 
\mathbf{O} = \mathcal{O}(G) \otimes \wedge(\mathfrak{g}_1^*)\to \mathcal{O}(G)$,
and 
$G$ is thus a closed super-subgroup of $\mathbf{\Gamma}$; see the paragraph following the proof. 
It follows that 
$\operatorname{Lie}(G)$ is a Lie super-subalgebra of $\operatorname{Lie}(\mathbf{\Gamma})$
through the inclusion
$\operatorname{Lie}(G) \to \operatorname{Lie}(G) \oplus \mathfrak{g}_1$.

It remains to compute $[v_1,v_2]$ in $\operatorname{Lie}(\mathbf{\Gamma})$, where $v_1,v_2 \in \mathfrak{g}_1$, or
$v_1 \in \mathfrak{g}_1$, $v_2 \in \operatorname{Lie}(G)$. 
If elements $\tau \in A$ and $v \in \operatorname{Lie}(\mathbf{\Gamma})$ satisfy $\tau^2=0$ and $|\tau|=|v|$,
then 
\[ g(\tau, v) : \mathbf{O} \to A,\quad g(\tau, v)(h)= \varepsilon(h)1+ \tau \, v(h),\; h \in \mathbf{O} \]
is an element in $\mathbf{\Gamma}(A)$ with inverse $g(-\tau, v)$. This coincides with $e(\tau, v)$ if
$|\tau|=|v|=1$. Note that $g(\tau, v)=i(f(\tau,x))$, if $|\tau|=0$ and $v=i'(x)$ with $x \in \mathfrak{g}_0$. 
Given elements $g_1=g(\tau_1, v_1)$, $g_2=g(\tau_2,v_2)$ as above, then
the commutator $(g_1,g_2)= g_1g_2g_1^{-1} g_2^{-1}$ coincides with 
\[
g((-1)^{|\tau_1||\tau_2|}\tau_1 \tau_2,\ [v_1,v_2]),
\] 
from which we will see the desired values of $[v_1,v_2]$. 

First, suppose that $A = \wedge(\tau_1, \tau_2)$, where $\tau_i$, $i=1,2$, are odd variables.
Let $u, v \in \mathfrak{g}_1$. Since we have $(e(\tau_1,u), e(\tau_2, v)) = g(-\tau_1\tau_2,\ i'([u,v]))$ 
by (i) of Lemma \ref{lem:relations}, it follows that
\[ 
[(0,u), (0,v)] = (i'([u,v]), 0). 
\]

Next, suppose that $A = \Bbbk[\tau_1]/(\tau_1^2) \otimes \wedge(\tau_2)$, where $\tau_1$ (resp., $\tau_2$)
is an even (resp., odd) variable. 
Let $y, v \in \operatorname{Lie}(\mathbf{\Gamma})$ with $y$ even and $v$ odd. Note
$g(\pm \tau_1,y)\in G(A_0)$. 
Since we see from (1) of Lemma \ref{lem:egfg} that 
\[ 
(g(-\tau_1,y), e(\tau_2,v))= e(\tau_2,v)^{g(\tau_1,y)}\, e(-\tau_2,v)= e(\tau_1\tau_2, \alpha'(y)(v)), 
\]
it follows that
\[ [(0,v), (y,0)] = (\alpha'(y)(v), 0). \]

(2)\ By Part 1 and the remark given above the proposition 
it remains to prove that the map preserves the $2$-operation. 
Suppose again that $A = \wedge(\tau_1, \tau_2)$. Then we see from (ii) of Lemma \ref{lem:relations} that
\[
g(-\tau_1\tau_2,\ i'(v^{\langle 2 \rangle}))  
= e(\tau_1,v)\, e(\tau_2,v)\, e(-(\tau_1+\tau_2), v).
\]
This last equals $g(-\tau_1\tau_2,\ v^2)$, which proves the desired result.  
\end{proof}

Recall that a \emph{closed super-subgroup} of an affine supergroup $\mathbf{G}$ is a subgroup functor
of $\mathbf{G}$ which is represented by a quotient Hopf superalgebra of $\mathcal{O}(\mathbf{G})$. 

\begin{rem}\label{rem:super_Lie_group}
We have worked in a general situation as above, aiming at an application to super Lie groups
over a complete field of characteristic $\ne 2$; see \cite{HM}. 
Suppose that $\mathfrak{G}$ is such a super Lie group. 
Then a Lie group, $\mathfrak{G}_{red}$, is naturally associated with it. 
Let $\mathcal{R}(\mathfrak{G}_{red})$ be the commutative Hopf algebra of all analytic representative
functions on $\mathfrak{G}_{red}$; this is not necessarily finitely generated. The corresponding
affine group and the Lie superalgebra $\operatorname{Lie}(\mathfrak{G})$ of $\mathfrak{G}$ 
have a natural pairing and an anti-morphism as in \eqref{eq:pairing_alpha}, which satisfy (D1)--(D3). 
In \cite{HM} the resulting affine supergroup $\mathbf{\Gamma}$ is used and is proved to be 
\emph{universal algebraic hull} of $\mathfrak{G}$ (see \cite[p.1141]{HochMostow}) in the sense that
$\mathbf{\Gamma}$-supermodules are naturally identified with analytic $\mathfrak{G}$-supermodules. 
\end{rem}

\section{The category equivalence over a commutative ring}\label{sec:equivalence}

We continue to suppose that $\Bbbk$ is an arbitrary non-zero commutative ring. 

\subsection{Re-proving Gavarini's equivalence}\label{subsec:reprove}
Let $\mathbf{G}$ be an affine supergroup, and set $\mathbf{O} = \mathcal{O}(\mathbf{G})$.   
Recall from \cite[Section 2.5]{MS}, for example, that 
the \emph{associated affine group} $\mathbf{G}_{ev}$ is the restricted group functor
$\mathbf{G}|_{\mathsf{Alg}_{\Bbbk}}$ defined on $\mathsf{Alg}_{\Bbbk}$. 
This is represented by
the largest purely even quotient Hopf superalgebra 
\begin{equation}\label{eq:bar_O}
\overline{\mathbf{O}} := \mathbf{O}/\mathbf{O}\mathbf{O}_1 \, (=\mathbf{O}_0/\mathbf{O}_1^2) 
\end{equation}
of $\mathbf{O}$, so that $\overline{\mathbf{O}}=\mathcal{O}(\mathbf{G}_{ev})$. 
This $\mathbf{G}_{ev}$ is also regarded as the closed super-subgroup of $\mathbf{G}$
which assigns to $A \in \mathsf{SAlg}_{\Bbbk}$ the group $\mathbf{G}(A_0)$.
Let 
\begin{equation}\label{eq:WO}
W^{\mathbf{O}}:= \mathbf{O}_1/\mathbf{O}_0^+\mathbf{O}_1,\ \text{where}\  
\mathbf{O}_0^+= \mathbf{O}_0 \cap \mathbf{O}^+, 
\end{equation}
as in \cite{M1}. 
Since $\mathbf{O}_0^+/((\mathbf{O}_0^+)^2+\mathbf{O}_1^2)
\simeq \overline{\mathbf{O}}^+/(\overline{\mathbf{O}}^+)^2$, we have
\[ \mathbf{O}^+/(\mathbf{O}^+)^2
\simeq \overline{\mathbf{O}}^+/(\overline{\mathbf{O}}^+)^2 \oplus W^{\mathbf{O}},
\]
which is dualized to
\[
\operatorname{Lie}(\mathbf{G})\simeq \operatorname{Lie}(\mathbf{G}_{ev}) \oplus (W^{\mathbf{O}})^*;
\]
see \cite[Lemma 4.3]{MS}. It follows that
\[
\operatorname{Lie}(\mathbf{G})_0 \simeq \operatorname{Lie}(\mathbf{G}_{ev}),\quad
\operatorname{Lie}(\mathbf{G})_1 = (W^{\mathbf{O}})^*.
\]
The former is the canonical Lie-algebra isomorphism induced from the embedding $\overline{\mathbf{O}}^*
\subset \mathbf{O}^{\bar{*}}$, through which we will identify as
\begin{equation*}\label{eq:identify} 
\operatorname{Lie}(\mathbf{G})_0 = \operatorname{Lie}(\mathbf{G}_{ev}).
\end{equation*} 

Just as for \eqref{eq:gc}, the right adjoint action 
$\mathbf{G} \times \mathbf{G}_{ev} \to \mathbf{G}$,\ $(f,g) \mapsto g^{-1}fg$
is dualized to the left $\mathbf{G}_{ev}$-supermodule structure on $\mathbf{O}$ defined by
\begin{equation}\label{eq:gh}
{}^g h = g^{-1}(h_{(1)})\, h_{(2)} \, g(h_{(3)}),\quad 
g \in \mathbf{G}_{ev}(R),\ h \in \mathbf{O},
\end{equation}
where $R \in \mathsf{Alg}_{\Bbbk}$. This makes $\mathbf{O}$ into a Hopf-algebra object in 
the symmetric tensor category $\mathbf{G}_{ev}\text{-}\mathsf{SMod}$ of left $\mathbf{G}_{ev}$-supermodules;
see Section \ref{subsec:G-supermodule}. 

Let us recall the definitions \cite[Definitions 3.2.6,\ 4.1.2]{G} of two categories,
following mostly the formulation of \cite[Appendix]{MS}. 

First, let $(\operatorname{gss}\text{-}\operatorname{fsgroups})_{\Bbbk}$ denote 
the category of the affine supergroups $\mathbf{G}$ such that
when we set $\mathbf{O} = \mathcal{O}(\mathbf{G})$, 
\begin{itemize}
\item[(E1)] $W^{\mathbf{O}}$ is $\Bbbk$-finite free,
\item[(E2)] $\overline{\mathbf{O}}^+/(\overline{\mathbf{O}}^+)^2$ is $\Bbbk$-finite projective, and
\item[(E3)] there exists a counit-preserving isomorphism 
$\mathbf{O}\simeq \overline{\mathbf{O}}\otimes \wedge(W^{\mathbf{O}})$ 
of left $\overline{\mathbf{O}}$-comodule superalgebras.
\end{itemize}
A morphism in $(\operatorname{gss}\text{-}\operatorname{fsgroups})_{\Bbbk}$ is a natural transformation of 
group functors. The conditions above re-number those (E1)--(E3) given in \cite[Appendix]{MS}.

\begin{rem}\label{rem:announce_TPD}  
It will be proved by Theorem \ref{thm:TPD} in the next subsection that an affine supergroup 
$\mathbf{G}$ necessarily satisfies (E3), if it satisfies (E1) and (E2), and if 
$\overline{\mathbf{O}}\, (= \mathcal{O}(\mathbf{G}_{ev}))$ is $\Bbbk$-flat.  
\end{rem}

Next, to define the category $(\operatorname{sHCP})_{\Bbbk}$, 
let $(G, \mathfrak{g})$ be a pair of an affine group $G$ and a Lie superalgebra $\mathfrak{g}$ 
equipped with a $2$-operation. Suppose that $\mathfrak{g}_1$ is $\Bbbk$-finite free, and 
is given a right $G$-module structure. 
Suppose in addition,  
\begin{itemize}
\item[(F1)] $\mathfrak{g}_0 = \operatorname{Lie}(G)$,
\item[(F2)] $\mathcal{O}(G)^+/(\mathcal{O}(G)^+)^2$ is $\Bbbk$-finite projective, so that   
$\mathfrak{g}_0 = \operatorname{Lie}(G)$ is necessarily 
$\Bbbk$-finite projective, and has the right $G$-module structure (see \eqref{eq:transpose}, 
and \eqref{eq:vg} below) determined by
\begin{equation}\label{eq:xg}
x^g(c) = x({}^gc),\quad x\in \mathfrak{g}_0,\ c \in \mathcal{O}(G), 
\end{equation}
where ${}^gc=g^{-1}(c_{(1)})\, c_{(2)}\, g(c_{(3)})$, as in \eqref{eq:gc},  
\item[(F3)] the left $\mathcal{O}(G)$-comodule structure $\mathfrak{g}_1 \to \mathcal{O}(G)\otimes \mathfrak{g}_1$,
$v \mapsto v_{(-1)}\otimes v_{(0)}$ on $\mathfrak{g}_1$ corresponding to the given right $G$-module structure
satisfies
\[
[v, x] = x(v_{(-1)})\, v_{(0)},\quad v \in \mathfrak{g}_1,\ x \in \mathfrak{g}_0,
\] 
\item[(F4)] the restricted super-bracket 
$[\ , \ ]|_{\mathfrak{g}_1 \otimes \mathfrak{g}_1} : \mathfrak{g}_1 \otimes \mathfrak{g}_1\to \mathfrak{g}_0$ 
is $G$-equivariant, and
\item[(F5)] the right $G$-module structure preserves the $2$-operation, or explicitly,
\[
(v_R^{\langle 2 \rangle})^g =(v^g)_R^{\langle 2 \rangle},\quad 
v \in (\mathfrak{g}_1)_R,\ g \in G(R),
\]
where $R \in \mathsf{Alg}_{\Bbbk}$.
\end{itemize}

As for the last (F5) we have replaced the original \cite[Appendix, (F5)]{MS} 
with the equivalent one given soon after. 

Finally, let $(\operatorname{sHCP})_{\Bbbk}$ denote the category of all those pairs $(G, \mathfrak{g})$ 
which satisfy (F1)--(F5) above. 
A morphism $(G, \mathfrak{g}) \to (G', \mathfrak{g}')$ in $(\operatorname{sHCP})_{\Bbbk}$
is a pair $(\gamma, \delta)$ of a morphism  
$\gamma : G \to G'$ of affine groups and a Lie superalgebra map 
$\delta = \delta_0 \oplus \delta_1 : \mathfrak{g} \to \mathfrak{g}'$, such that 
\begin{itemize}
\item[(F6)] the Lie algebra map $\mathrm{Lie}(\gamma)$ induced from $\gamma$ coincides with $\delta_0$,  
\item[(F7)] 
$(\delta_1)_R(v^g)= \delta_1(v)^{\gamma_{_R}(g)}$,\; $v \in \mathfrak{g}_1$,\ $g \in G(R)$,\\
where $R\in \mathsf{Alg}_{\Bbbk}$,\ and
\item[(F8)] $\delta_0(v^{\langle 2 \rangle}) = \delta_1(v)^{\langle 2 \rangle}$,\; $v \in \mathfrak{g}_1$.
\end{itemize} 

Let us reproduce from \cite{G} functors between the two categories just defined,
\[ 
\Phi : (\operatorname{gss}\text{-}\operatorname{fsgroups})_{\Bbbk} \to (\operatorname{sHCP})_{\Bbbk},
\quad
\Psi : (\operatorname{sHCP})_{\Bbbk}\to (\operatorname{gss}\text{-}\operatorname{fsgroups})_{\Bbbk},
\]
which are denoted by $\Phi_g$, $\Psi_g$ in \cite{G}.  

First, let $\mathbf{G}$ be an object in $(\operatorname{gss}\text{-}\operatorname{fsgroups})_{\Bbbk}$. 
Set $\mathbf{O}=\mathcal{O}(\mathbf{G})$. Consider
the pair
\[
(G, \mathfrak{g}) :=(\mathbf{G}_{ev}, \operatorname{Lie}(\mathbf{G})), 
\]
giving to $\mathfrak{g}_1$ the right $G$-module structure determined by 
\begin{equation}\label{eq:vg}
v^g(h) = v({}^gh),\quad v \in \mathfrak{g}_1,\ h\in \mathbf{O},\ g \in G(R), 
\end{equation}
where $R \in \mathsf{Alg}_{\Bbbk}$, and ${}^gh$ is given by \eqref{eq:gh}. 
To see that this indeed defines a right $G$-module structure, note that
the left $G$-module structure on $\mathbf{O}$ given by \eqref{eq:gh} induces such a structure on
$\mathbf{O}^+/(\mathbf{O}^+)^2$, and the induced structure is transposed to $\mathfrak{g}$, 
since $\mathbf{O}^+/(\mathbf{O}^+)^2$ is $\Bbbk$-finite projective by (E1)--(E2); see \eqref{eq:transpose}. 
What is given by \eqref{eq:vg} is precisely the restriction to $\mathfrak{g}_1$
of the transposed structure, while the restriction to $\mathfrak{g}_0$ coincides with the one given by
\eqref{eq:xg}. It is now easy to see that the pair satisfies (F1)--(F4).
Recall that $\mathfrak{g}$ is equipped with the $2$-operation which arises from the square map on $\mathbf{O}^{\bar{*}}$. 
Then one verifies (F5), using the fact that the $G$-module structure on $\mathbf{O}$ 
preserves the coproduct; cf. \cite[Lemma A.9]{MS}. 
Therefore, $(G, \mathfrak{g})\in (\operatorname{sHCP})_{\Bbbk}$. We let
\[
\Phi(\mathbf{G}) = (\mathbf{G}_{ev}, \operatorname{Lie}(\mathbf{G})). 
\] 
One sees easily that this indeed defines a functor. 

\begin{rem}\label{rem:larger_category}  
Following \cite[Defintion 2.3.3]{G}, let 
$(\operatorname{fsgroups})_{\Bbbk}$ 
denote the category of those affine supergroup which satisfy
(E1) and (E2). This includes 
$(\operatorname{gss}\text{-}\operatorname{fsgroups})_{\Bbbk}$
as a full subcategory. 
Note that Condition (E3) was not used above, to define the functor $\Phi$.  
In fact we have thus defined a
functor $\Phi  : (\operatorname{fsgroups})_{\Bbbk} \to (\operatorname{sHCP})_{\Bbbk}$, 
as is formulated by \cite[Proposition 4.1.3]{G}. 
This last functor will be used to prove Theorem \ref{thm:TPD} in the next subsection. 
\end{rem}

Next, to construct $\Psi$, we prove:

\begin{lemma}\label{lem:Gamma}
Let $\mathbf{\Gamma}$ be the affine supergroup constructed in Section \ref{sec:construction},
and set $\mathbf{O}=\mathcal{O}(\mathbf{\Gamma})$. Then we have
\[
\mathbf{\Gamma}_{ev} = G, \quad W^{\mathbf{O}} = \mathfrak{g}_1^*,
\]
where $G$ and $\mathfrak{g}$ are those given in Section \ref{sec:construction} from which 
$\mathbf{\Gamma}$ is constructed.
Moreover, $\mathbf{\Gamma}$ satisfies (E1) and (E3) above. 
\end{lemma}
\begin{proof}
From Remark \ref{rem:O} and the following argument we see that 
\eqref{eq:OGamma} gives an identification 
$\mathcal{O}(\mathbf{\Gamma}) = \mathcal{O}(G) \otimes \wedge(\mathfrak{g}_1^*)$ 
of left $\mathcal{O}(G)$-comodule superalgebras with counit. This implies the desired results.  
\end{proof}

Finally, let $(G, \mathfrak{g}) \in (\operatorname{sHCP})_{\Bbbk}$. 
Choose these $G$ and $\mathfrak{g}$
as those in Section \ref{sec:construction}. 
By (F1)--(F2), $\mathfrak{g}$ satisfies (C); 
see Proposition \ref{prop:PBW}. 
The given right $G$-module structure on $\mathfrak{g}_1$, summed up with 
such a structure on $\mathfrak{g}_0$ determined by \eqref{eq:xg}, gives rise to an 
anti-morphism, say $\alpha$, from $G$ to $\mathsf{Aut}_{Lie}(\mathfrak{g})$; see \cite[Remark 4.5 (2)]{MS}. 
This $\alpha$, together with the canonical pairing
\[ 
\langle \ , \ \rangle : \mathfrak{g}_0 \times \mathcal{O}(G) \to \Bbbk, \quad \langle x,\ c \rangle = x(c),
\]
satisfy (D1)--(D3), as is easily seen. We remark that Lie algebra map 
$i' : \mathfrak{g}_0 \to \operatorname{Lie}(G)$ given by
\eqref{eq:i_Lie_map} is now the identity.
The construction of Section \ref{sec:construction} gives an affine supergroup $\mathbf{\Gamma}$,
which satisfies (E1)--(E3) by Lemma \ref{lem:Gamma}. 
Indeed, by (F2) it satisfies (E2) as well, since 
$\mathbf{\Gamma}_{ev}=G$. Define $\Psi(G, \mathfrak{g})$ to be this $\mathbf{\Gamma}$ in 
$(\operatorname{gss}\text{-}\operatorname{fsgroups})_{\Bbbk}$. As is easily seen,
$\Psi$ defines a functor. 

\begin{theorem}[\text{\cite[Theorem~4.3.14]{G}}]\label{thm:Gavarini's_equivalence}
We have a category equivalence 
\[
(\operatorname{gss}\text{-}\operatorname{fsgroups})_{\Bbbk} \approx (\operatorname{sHCP})_{\Bbbk}.
\]
In fact the functors
$\Phi$ and $\Psi$ constructed above are quasi-inverse to each other. 
\end{theorem}
\begin{proof}
Let $\mathbf{G} \in (\operatorname{gss}\text{-}\operatorname{fsgroups})_{\Bbbk}$, and set
\[ 
(G, \mathfrak{g}) = \Phi(\mathbf{G}),\quad \mathbf{\Gamma} = \Psi \circ \Phi(\mathbf{G}).
\]
Just as for \eqref{eq:Hopf_pair} we see that there uniquely exists a Hopf paring 
\[ 
\langle \ , \ \rangle : \mathbf{U}(\mathfrak{g}) \times \mathcal{O}(\mathbf{G}) \to \Bbbk 
\]
such that $\langle z,\ h \rangle = z(h)$, $z \in \mathfrak{g}$, $h \in \mathcal{O}(\mathbf{G})$.  
Suppose $A \in \mathsf{SAlg}_{\Bbbk}$. 
Recall that $\mathbf{\Gamma}(A)$ is a quotient of the group $G(A_0) \ltimes \mathbf{\Sigma}(A)$ 
of semi-direct product. 
Since $\mathbf{\Sigma}(A) \subset \mathsf{Gpl}(\mathbf{U}(\mathfrak{g})_A)$,
the last pairing induces, after base extension to $A$, a group map
\begin{equation}\label{eq:group_map_from_Sigma}
\mathbf{\Sigma}(A) \to \mathsf{SAlg}_{\Bbbk}(\mathcal{O}(\mathbf{G}), A)=\mathbf{G}(A). 
\end{equation}
Lemma \ref{lem:egfg} gives the following equations in $\mathbf{\Sigma}(A)$:
\begin{equation}\label{eq:eavg}
e(a, v)^g=1 \otimes 1 + a\, v^g,\ 
f(\epsilon, x)^g = 1\otimes 1 + \epsilon\, x^g, \quad g \in G(A_0).
\end{equation}
By definitions of $\Phi$ and $\Psi$, the $G$-actions on $\mathfrak{g}$ which appear 
on the right-hand sides are determined by
\[ \langle z^g,\ h\rangle_{A_0} =\langle z,\ {}^gh\rangle_{A_0},\quad z \in \mathfrak{g},\ h \in \mathcal{O}(\mathbf{G}),\ g \in G(A_0), \]
where ${}^gh= g^{-1}(h_{(1)})\, h_{(2)}\, g(h_{(3)})$, as in \eqref{eq:gh}. It follows that the group map 
\eqref{eq:group_map_from_Sigma} is right $G(A_0)$-equivariant, where we suppose that 
$G(A_0)=\mathbf{G}(A_0)$ acts on 
$\mathbf{G}(A)$ by inner automorphisms. 
Therefore, the group map together with the embedding
$G(A_0) \to \mathbf{G}(A)$
uniquely extend to
$G(A_0) \ltimes \mathbf{\Sigma}(A) \to \mathbf{G}(A)$. It factors through 
$\mathbf{\Gamma}(A) \to \mathbf{G}(A)$, since
$\mathbf{\Gamma}(A)$ is the quotient group of $G(A_0) \ltimes \mathbf{\Sigma}(A)$ divided by the relations
\[ (i(f(\epsilon, x)), 1)=(1,f(\epsilon, x)), \quad x \in \mathfrak{g}_0, \ \epsilon \in A_0,\ \epsilon^2=0, \]
and $i : F(A_0) \to G(A_0)$ is now the restriction of the canonical map 
$\mathsf{Gpl}(U(\mathfrak{g}_0)_{A_0}) \to G(A_0)$. 
The group map $\mathbf{\Gamma}(A) \to \mathbf{G}(A)$, being natural in $A$, gives rise to a morphism
$\mathbf{\Gamma} \to \mathbf{G}$.
This morphism is natural in $\mathbf{G}$, as is easily seen. 
In fact, it is a natural isomorphism  
by \cite[Lemma 4.26]{MS}; see also Remark \ref{rem:lemma4.26} below. 
Indeed,
the assumptions required by the cited lemma are satisfied, since $\mathbf{\Gamma}$ and $\mathbf{G}$ satisfy (E3), 
the morphism $\mathbf{\Gamma} \to \mathbf{G}$ restricts to the identity $\mathbf{\Gamma}_{ev} \to \mathbf{G}_{ev}$, 
and the map $\mathfrak{g}_1=\operatorname{Lie}(\mathbf{\Gamma})_1\to \operatorname{Lie}(\mathbf{G})_1$
induced from the pairing above is the identity. 
We conclude $\Psi \circ \Phi \simeq \operatorname{id}$.

Let $(G, \mathfrak{g}) \in (\operatorname{sHCP})_{\Bbbk}$, and set $\mathbf{\Gamma} = \Psi(G, \mathfrak{g})$.
Recall that for this $\mathbf{\Gamma}$, 
the Lie algebra map $i' : \mathfrak{g}_0 \to \operatorname{Lie}(G)$ given by
\eqref{eq:i_Lie_map} is the identity.
By Lemma \ref{lem:Gamma} and Proposition \ref{prop:bracket} we have the natural identifications
\[ G = \mathbf{\Gamma}_{ev}, \quad \mathfrak{g} = \operatorname{Lie}(\mathbf{\Gamma}) \]
of affine groups and of Lie superalgebras equipped with $2$-operation.
Let $R \in \mathsf{Alg}_{\Bbbk}$.
To conclude $\Phi \circ \Psi = \operatorname{id}$, we wish to prove that given 
$v \in \mathfrak{g}_1$ and $g \in G(R)$, the result $v^g\in (\mathfrak{g}_1)_R$ by the 
$G$-action associated with the original 
$(G, \mathfrak{g})$ coincides the one given by \eqref{eq:vg} for $\mathbf{\Gamma}$.
Suppose $A = R \otimes \wedge(\tau)$, where $\tau$ is an odd variable. Note
$A_0 =R$. Just as in \eqref{eq:eavg} we have 
$e(\tau, v)^g=1\otimes 1 + \tau \, v^g$ in $\mathbf{\Gamma}(A)$. 
This, evaluated at $h \in \mathcal{O}(\mathbf{\Gamma})$, gives $\tau\, v(^{g}h)=\tau\, v^g(h)$, 
which shows the desired result. 
\end{proof}

\begin{rem}\label{rem:compare_with_MS}
Let $(G, \mathfrak{g}) \in (\operatorname{sHCP})_{\Bbbk}$, and recall that this $\mathfrak{g}$
is equipped with a $2$-operation, say $(\ )^{\langle 2 \rangle}$. 
Replace 
$(\mathfrak{g}, (\ )^{\langle 2 \rangle})$ with the cocycle deformation 
$({}_{\sigma}\mathfrak{g}, (\ )^{{}_{\sigma}\hspace{-0.5mm}\langle 2 \rangle})$ by $\sigma$
(see Remark \ref{rem:deform_Lie}),
keeping the right $G$-module structure on the odd component unchanged. 
Then we see $(G, {}_{\sigma}\mathfrak{g}) \in (\operatorname{sHCP})_{\Bbbk}$, and that
$(G, \mathfrak{g})\mapsto (G, {}_{\sigma}\mathfrak{g})$ gives an involutive category isomorphism, which
we denote by 
\[
(\operatorname{id}, {}_{\sigma}(\ )) : (\operatorname{sHCP})_{\Bbbk}\to (\operatorname{sHCP})_{\Bbbk}.
\]

As was remarked in Introduction,
Gavarini's category equivalence was re-proved also in \cite[Appendix]{MS},
using an older construction of affine supergroups.  
Due to different choice of tensor products of pairings, the category equivalence
$\mathbf{P}' : (\operatorname{gss}\text{-}\operatorname{fsgroups})_{\Bbbk}\to (\operatorname{sHCP})_{\Bbbk}$
obtained there is slightly different from the $\Phi$ above. In fact, we see 
\begin{equation}\label{eq:up_to_isom}
\mathbf{P}'= (\operatorname{id}, {}_{\sigma}(\ ))\circ \Phi.
\end{equation}
\end{rem}

\begin{rem}\label{rem:compare_with_G} 
The argument of Gavarini \cite{G} seems incomplete at some points, as is pointed out below.
See also \cite[Remark A.11]{MS}. 

(1)\
To construct the functor 
$\Phi_g : (\operatorname{gss}\text{-}\operatorname{fsgroups})_{\Bbbk}\to (\operatorname{sHCP})_{\Bbbk}$, 
and prove $\Phi_g \circ \Psi_g = \operatorname{id}$
in \cite[Proposition 4.1.3, Theorem 4.3.14]{G}, the article takes no account of
$2$-operations or $G$-supermodule structures on Lie superalgebras.

(2)\
The functoriality of $\Psi_g : (\operatorname{sHCP})_{\Bbbk}
\to (\operatorname{gss}\text{-}\operatorname{fsgroups})_{\Bbbk}$ (see \cite[Proposition 4.3.9~(2)]{G})
is proved, indeed,
if one replaces the original definition of $\Psi_g$ by the group 
$G_{\mathcal{P}}(A)\, (= \Psi_g(\mathcal{P}))$ given in 
\cite[Definition 4.3.2]{G} (and referred to before Lemma \ref{lem:free_product}),
with the definition by the alternative $G_{\mathcal{P}}^{\bullet}(A)$ given in \cite[Remark 4.3.3~(c)]{G}. 
Nevertheless, in view of the equations
preceding our Lemma \ref{lem:e(a,v)}, the relation 
$(1+(c\eta) Y)=(1+\eta(cY))$, $c \in \Bbbk$, is missing to define the group 
$G_{\mathcal{P}}^{\bullet}(A)$ in the last cited remark.
\end{rem}

\subsection{Tensor product decomposition}\label{subsec:TPD} 
Let $\mathbf{G}$ be an affine supergroup, and set $\mathbf{O} = \mathcal{O}(\mathbf{G})$.
Recall from \eqref{eq:bar_O} and \eqref{eq:WO} the definitions of $\overline{\mathbf{O}}$ and $W^{\mathbf{O}}$. 
As was announced in Remark \ref{rem:announce_TPD} we prove the following theorem. 
Note that the conclusion below is the same as (E3). 

\begin{theorem}\label{thm:TPD} 
Assume that $\overline{\mathbf{O}}\, (=\mathcal{O}(\mathbf{G}_{ev}))$ is $\Bbbk$-flat. Then there exists a counit-preserving isomorphism 
$\mathbf{O}\simeq \overline{\mathbf{O}}\otimes \wedge(W^{\mathbf{O}})$ 
of left $\overline{\mathbf{O}}$-comodule superalgebras, if 
\begin{itemize}
\item[(E1)] $W^{\mathbf{O}}$ is $\Bbbk$-finite free, and
\item[(E2)] $\overline{\mathbf{O}}^+/(\overline{\mathbf{O}}^+)^2$ is $\Bbbk$-finite projective. 
\end{itemize}
\end{theorem}

\begin{rem}\label{rem:compare_TPD}
(1)\
Let $(\operatorname{gss}\text{-}\operatorname{fsgroups})_{\Bbbk}'$ denote the category of the affine
supergroups $\mathbf{G}$ which satisfy (E1), (E2) and 
\begin{itemize}
\item[(E0)] $\mathcal{O}(\mathbf{G}_{ev})$ is $\Bbbk$-flat.
\end{itemize}
This category is a full subcategory of $(\operatorname{gss}\text{-}\operatorname{fsgroups})_{\Bbbk}$
by Theorem \ref{thm:TPD}. Let $(\operatorname{sHCP})_{\Bbbk}'$ denote the full subcategory of 
$(\operatorname{sHCP})_{\Bbbk}$ which consists of the objects $(G, V)$ such that
\begin{itemize}
\item[(F0)] $\mathcal{O}(G)$ is $\Bbbk$-flat.
\end{itemize}
One sees that the category equivalence given by Theorem \ref{thm:Gavarini's_equivalence} restricts to
\[
(\operatorname{gss}\text{-}\operatorname{fsgroups})_{\Bbbk}' \approx (\operatorname{sHCP})_{\Bbbk}'.
\]

(2)\ Suppose that $\Bbbk$ is $2$-torsion free, or namely, $2 : \Bbbk \to \Bbbk$ is an injection. 
In this special situation, essentially the same category equivalence as given by 
Theorem \ref{thm:Gavarini's_equivalence}
was proved by \cite[Theorem 4.22]{MS}; one need not there refer to
$2$-operations. To be more precise, considered there is the category $\mathsf{ASG}$ of the algebraic 
supergroups $\mathbf{G}$ which satisfy (E0) as well as (E1)--(E3); see \cite[Section 4.3]{MS}. However, 
(E3) can be removed from the last conditions, since it is ensured by Theorem \ref{thm:TPD}. To define 
$\mathsf{ASG}$ in \cite{MS}, one can thus weaken the condition that $\mathbf{O}=\mathcal{O}(\mathbf{G})$ is 
\emph{split}
\cite[Definition 2.1]{MS} to the one that $W^{\mathbf{O}}$ is $\Bbbk$-free. 
See \cite[Note added in proof]{MS}.

(3)\ Suppose that $\Bbbk$ is a field of characteristic $\ne 2$. Then the conclusion of Theorem
\ref{thm:TPD} holds for any finitely generated super-commutative Hopf superalgebra $\mathbf{O}$, since
the assumptions are then necessarily satisfied. The result was in fact proved by \cite[Theorem 4.5]{M1} for any 
$\mathbf{O}$ that is not necessarily finitely generated. The proof uses Hopf crossed products,
and is crucial when $\mathbf{O}$ is finitely generated. The proof below gives an alternative
proof of the cited theorem in this crucial case. 
\end{rem}

The rest of this subsection is devoted to proving the theorem. The proof is divided into 3 steps.

\subsubsection{Step 1}
Recall from Remark \ref{rem:larger_category} that the functor $\Phi$ is defined
on the category $(\operatorname{fsgroups})_{\Bbbk}$ 
including $(\operatorname{gss}\text{-}\operatorname{fsgroups})_{\Bbbk}$,
which consists of the affine supergroup satisfying (E1) and (E2).

Let $\mathbf{G} \in (\operatorname{fsgroups})_{\Bbbk}$, and set 
$\mathbf{\Gamma}= \Psi\circ \Phi (\mathbf{G})$,
as in the proof of Theorem \ref{thm:Gavarini's_equivalence}. 
The argument in the cited proof which shows that 
we have a natural morphism $\mathbf{\Gamma} \to \mathbf{G}$ of affine supergroups is valid. 
Let
\begin{equation}\label{eq:phi}  
\phi : \mathbf{\Gamma} \to \mathbf{G}
\end{equation}
denote the morphism. We will prove that  
this $\phi$ is an isomorphism, assuming that $\overline{\mathbf{O}}$ is $\Bbbk$-flat. 
This proves the theorem, since $\mathbf{\Gamma}$ satisfies (E3).

\subsubsection{Step 2}\ We need some general Hopf-algebraic argument. 
Let $\mathbb{N}= \{ 0, 1, 2, \dots \}$ denote the semigroup of non-negative integers. 
An $\mathbb{N}$-graded $\Bbbk$-module $V = \bigoplus_{n=0}^{\infty} V(n)$ is regarded as a $\Bbbk$-supermodule
so that $V_0 = \bigoplus_{n\ \text{even}} V(n)$,\ $V_1 = \bigoplus_{n\ \text{odd}} V(n)$. 
The $\mathbb{N}$-graded $\Bbbk$-modules form a symmetric tensor category 
$\mathsf{GrMod}_{\Bbbk}$ with respect to the super-symmetry.  

Let $\mathsf{ConnAlg}_{\Bbbk}$ denote the category of the commutative algebra objects $\mathbf{B}$ in 
$\mathsf{GrMod}_{\Bbbk}$ such that $\mathbf{B}(0)=\Bbbk$; 
the $\mathsf{Conn}$ expresses ``connected", meaning $\mathbf{B}(0)=\Bbbk$.

Fix a commutative Hopf algebra $O$.
Note that $O$ is a commutative Hopf-algebra object in $\mathsf{GrMod}_{\Bbbk}$ 
which is trivially graded, $O(0)=O$. 
A \emph{graded left $O$-comodule} is a left $O$-comodule object in $\mathsf{GrMod}_{\Bbbk}$. 
The graded left $O$-comodules form a symmetric tensor category $O\text{-}\mathsf{GrComod}$. 
Let $O\text{-}\mathsf{NGrComodAlg}$ denote the category of the commutative algebra objects
$\mathbf{A}$ in $O\text{-}\mathsf{GrComod}$ such that $\mathbf{A}(0)=O$;
the $\mathsf{NGr}$ expresses ``neutrally graded", meaning $\mathbf{A}(0)=O$. 
Note that every
such object is an (ordinary) left $O$-Hopf module \cite[Page 15]{Mo}
with respect to the left multiplication by $O$. 

Here, commutative algebra objects may not satisfy the condition 
that every odd elements should be square-zero. 

Given $\mathbf{B} \in \mathsf{ConnAlg}_{\Bbbk}$, the tensor product 
\[ 
O \otimes \mathbf{B}
\]
of graded algebras, given the left $O$-comodule structure $\Delta \otimes \mathrm{id}_{\mathbf{B}}$, 
is an object
in $O\text{-}\mathsf{NGrComodAlg}$. 
Moreover, this constructs a functor 
\[ O\otimes : \mathsf{ConnAlg}_{\Bbbk} \to O\text{-}\mathsf{NGrComodAlg}. \] 

\begin{prop}\label{prop:O_equivalence}
This functor is a category equivalence.
\end{prop}

\begin{proof} 
Given $\mathbf{A} \in O\text{-}\mathsf{NGrComodAlg}$,
\[ \mathbf{A}/O^+\mathbf{A} \]
is naturally an object in $\mathsf{ConnAlg}_{\Bbbk}$. One sees that this constructs
a functor. We wish to show that this is a quasi-inverse of the functor $O \otimes$.
We have to prove that the two composites of the functors are naturally isomorphic to the identity functors.
For one composite this is easy. For the remaining, let $\mathbf{A} \in O\text{-}\mathsf{NGrComodAlg}$.
Set $\mathbf{B}= \mathbf{A}/O^+\mathbf{A}$, and let $\pi : \mathbf{A} \to \mathbf{B}$ denote the
natural projection. We see that the left $O$-comodule structure
$ \mathbf{A} \to O \otimes \mathbf{A},\ a \mapsto a_{(-1)}\otimes a_{(0)} $
on $\mathbf{A}$ induces the morphism  
\begin{equation}\label{eq:Hopf_module_isom}
\mathbf{A} \to O \otimes \mathbf{B},\quad a \mapsto a_{(-1)} \otimes \pi(a_{(0)})
\end{equation}
in $O\text{-}\mathsf{NGrComodAlg}$ which is natural in $\mathbf{A}$. 
It remains to prove that this is an isomorphism.
As was remarked before, $\mathbf{A}$ is a left $O$-Hopf module, and the morphism above is 
in fact a morphism of Hopf modules. 
The fundamental theorem
for Hopf modules \cite[1.9.4, Page 15]{Mo} holds over an arbitrary base ring $\Bbbk$,
and can now apply to see that \eqref{eq:Hopf_module_isom} is an isomorphism; 
explicitly, the inverse is induced from $O\otimes\mathbf{A}\to \mathbf{A}$,
$c\otimes a \mapsto cS(a_{(-1)})a_{(0)}$.
\end{proof} 

Let $\mathbf{O}$ be a super-commutative Hopf superalgebra.
Set $O= \overline{\mathbf{O}}$, and assume that this $O$ is $\Bbbk$-flat. 
Let $O\text{-}\mathsf{SComod}$ denote the symmetric tensor category of
left $O$-super-comodules. The flatness assumption ensures that this category is abelian; 
see \cite[Part I, 2.9]{J}.
Indeed, the $\Bbbk$-linear kernel $Z$ of a morphism $V \to U$ turns to be a sub-object of $V$, since
we have $O \otimes Z \subset O\otimes V$, and the composite 
$Z \hookrightarrow V \to O \otimes V$ of the inclusion with the structure on $V$
factors through $O \otimes Z$. 

Let $I =\mathbf{O}\mathbf{O}_1$, so that we have $\mathbf{O}/I=O$. 
Note that $\mathbf{O}$
is naturally a commutative algebra object in $O\text{-}\mathsf{SComod}$, and the super-ideals
$I^n$, $n>0$, are sub-objects of $\mathbf{O}$ in $O\text{-}\mathsf{SComod}$. It follows that 
\[ \mathrm{gr}\, \mathbf{O} = \bigoplus_{n=0}^{\infty} I^n/I^{n+1} \]
is an object in $O\text{-}\mathsf{NGrComodAlg}$. 
To see this, note $\mathrm{gr}\, \mathbf{O}(0)= O$. Moreover, 
$I^n/I^{n+1}= \mathbf{O}_1^n/\mathbf{O}_1^{n+2}$, and so
$\mathrm{gr}\, \mathbf{O}(n)$ is purely odd (resp., even) if $n$ is odd (resp., even).

Let $\mathbf{B} = \mathrm{gr}\, \mathbf{O}/O^+(\mathrm{gr}\, \mathbf{O})$ denote the object in 
$\mathsf{ConnAlg}_{\Bbbk}$ which corresponds to $\mathrm{gr}\, \mathbf{O}$
through the category equivalence given in (the proof of) Proposition \ref{prop:O_equivalence}. 
It is easy to see the following (see \cite[Proposition 4.3 (1)]{M1}):

\begin{lemma}\label{lem:composite_isom}
The composite of natural maps 
\[
W^{\mathbf{O}}=\mathbf{O}_1/\mathbf{O}_0^+\mathbf{O}_1 \to 
\mathbf{O}_1/\mathbf{O}_1^3=\mathrm{gr}\, \mathbf{O}(1) \to \mathbf{B}(1)
\] 
is an isomorphism.
\end{lemma}

\subsubsection{Step 3}
Let $\mathbf{O}$ be a super-commutative Hopf superalgebra.
Note that the constructions of the associated $\overline{\mathbf{O}}$ and $W^{\mathbf{O}}$ are functorial. 

Assume that $\mathbf{O}$ satisfies (E1) and (E3). 
Assume that $\overline{\mathbf{O}}$ is $\Bbbk$-flat. 
Let $\mathbf{O}'$ be a super-commutative Hopf superalgebra, and let 
$\psi : \mathbf{O}' \to \mathbf{O}$ is a Hopf superalgebra map. It naturally induces
\[ \overline{\psi} : \overline{\mathbf{O}'}\to \overline{\mathbf{O}}, \quad
W^{\psi} : W^{\mathbf{O}'}\to W^{\mathbf{O}}.
\]  

\begin{prop}\label{prop:psi_isom}
If these two maps are bijections, then $\psi$ is an isomorphism.
\end{prop}

\begin{proof}
We may suppose $\overline{\mathbf{O}'}= \overline{\mathbf{O}}= O$ and $\overline{\psi} =\mathrm{id}_O$,
where $O$ is a commutative $\Bbbk$-flat Hopf algebra. We see that $\psi$ induces a 
morphism $\mathrm{gr}(\psi): \mathrm{gr}\, \mathbf{O}' \to \mathrm{gr}\, \mathbf{O}$ 
in $O\text{-}\mathsf{NGrComodAlg}$.
Let 
$\xi : \mathbf{B}'\to \mathbf{B}$ be the corresponding morphism between the corresponding objects in 
$\mathsf{ConnAlg}_{\Bbbk}$. 

We wish to show that $\xi$ is an isomorphism. 
By Lemma \ref{lem:composite_isom}, $\xi(1) : \mathbf{B}'(1) \to \mathbf{B}(1)$
is identified with $W^{\psi}$. 
Since $\mathbf{O}$ satisfies (E3), we see that 
$\mathrm{gr}\, \mathbf{O}=O \otimes \wedge(W^{\mathbf{O}})$, and so $\mathbf{B}=\wedge(W^{\mathbf{O}})$. 
It follows that $\xi$ has a unique section in $\mathsf{ConnAlg}_{\Bbbk}$, since $\xi(1)$ is an isomorphism,
and $\mathbf{B}'$ is super-commutative, with the odd elements being square-zero. 
Note that $\mathbf{B}'$ is generated by $\mathbf{B}'(1)$, since $\mathrm{gr}\, \mathbf{O}'$
is generated by $O= \mathrm{gr}\, \mathbf{O}'(0)$ and $\mathrm{gr}\, \mathbf{O}'(1)$.
This implies that the section is an isomorphism, proving the desired result.

It follows that $\mathrm{gr}(\psi)$ is an isomorphism, and 
$\mathrm{gr}\, \mathbf{O}'(n)=
\mathrm{gr}\, \mathbf{O}(n)=0$ for $n \gg 0$. Therefore, $\psi$ is an isomorphism. 
\end{proof}

\begin{rem}\label{rem:lemma4.26}
In the situation of Proposition \ref{prop:psi_isom}, suppose in addition that $\mathbf{O}'$ satisfies (E3), 
and remove the assumption that $\overline{\mathbf{O}}$ is $\Bbbk$-flat. Then the same result as
the proposition follows easily from \cite[Lemma 4.26]{MS}. The result was essentially used to prove
\cite[Theorem A.11]{MS} in the last paragraph of the proof.   
\end{rem}

Let us return to the natural morphism $\phi : \mathbf{\Gamma} \to \mathbf{G}$ in 
\eqref{eq:phi}, assuming that
$\overline{\mathcal{O}(\mathbf{G})}$ is $\Bbbk$-flat. Consider 
$\mathcal{O}(\phi) : \mathcal{O}(\mathbf{G})\to \mathcal{O}(\mathbf{\Gamma})$. 
In view of the proof of Theorem \ref{thm:Gavarini's_equivalence} (see the last part of the first paragraph),
the induced
$\overline{\mathcal{O}(\mathbf{G})} \to \overline{\mathcal{O}(\mathbf{\Gamma})}$ and 
$W^{\mathcal{O}(\mathbf{G})} \to W^{\mathcal{O}(\mathbf{\Gamma})}$ are both the identity maps. 
It follows that $\overline{\mathcal{O}(\mathbf{\Gamma})}$ is $\Bbbk$-flat.
Since $\mathbf{\Gamma}$ satisfies (E1) and (E3), Proposition \ref{prop:psi_isom}, 
applied to $\mathcal{O}(\phi)$, proves
that $\phi$ is an isomorphism, as desired.

\section{The category equivalence over a field}\label{sec:equiv_over_fields}

In this section we suppose that $\Bbbk$ is a field of characteristic $\ne 2$. 

\subsection{Reformulation}\label{subsec:equiv_over_fields}
We let $\mathsf{ASG}_{\Bbbk}$ denote the category of algebraic supergroups over $\Bbbk$.
This coincides with the full subcategory of $(\operatorname{gss}\text{-}\operatorname{fsgroups})_{\Bbbk}$
consisting of the objects which are algebraic supergroups.  By \cite[Theorem 4.5]{M1}
(or Theorem \ref{thm:TPD} above)
every object in $\mathsf{ASG}_{\Bbbk}$ satisfies (E3), in particular; see Remark \ref{rem:compare_TPD}~(3). 

Note from Remark \ref{rem:2-torsion_free} (2) that 
we may not refer to $2$-operations on Lie superalgebras. 
The definition of $(\operatorname{sHCP})_{\Bbbk}$ then
contains redundancy in (F1). In other words one can remove $\mathfrak{g}_0$ from the definition
since it is determined by $G$. 
Following \cite{M3, MZ}, we modify the definition of Harish-Chandra pairs as follows; 
one will see in the next subsection
that this modified definition is suitable at least to describe sub-objects. 

A \emph{Harish-Chandra pair} is a pair $(G,V)$ of an algebraic group $G$ and a finite-dimensional
right $G$-module $V$ which is equipped with a $G$-equivariant linear map $[\ , \ ] : V \otimes V \to \mathrm{Lie}(G)$ 
such that 
\begin{itemize}
\item[(G1)] $[v,v']=[v',v]$,\quad $v, v' \in V$,  
\item[(G2)] $v \triangleleft [v,v]= 0$,\quad $v \in V$.  
\end{itemize}
When we say that $[\ , \ ]$ is $G$-equivariant, $\mathrm{Lie}(G)$ is regarded as
a right $G$-module as was done in \eqref{eq:xg}. In (G2), $\triangleleft$ represents the right 
$\mathrm{Lie}(G)$-Lie module structure on $V$ defined by
\begin{equation}\label{eq:v_triangleleft_x} 
v\triangleleft x = x(v_{(-1)})\, v_{(0)}, \quad v \in V,\ x \in \mathrm{Lie}(G), 
\end{equation}
where $V \to \mathcal{O}(G) \otimes V$,\ $v \mapsto v_{(-1)} \otimes v_{(0)}$ denotes the left
$\mathcal{O}(G)$-comodule structure corresponding to the right $G$-module structure on $V$. 
A \emph{morphism} $(\phi, \psi) : (G_1,V_1) \to (G_2,V_2)$ of Harish-Chandra pairs consists 
of a morphism $\phi : G_1 \to G_2$ of algebraic groups and a linear map $\psi : V_1\to V_2$ such that
\begin{itemize}
\item[(G3)] $\psi$ is $G_1$-equivariant, with $V_2$ regarded as a $G_1$-module through $\phi$,
\item[(G4)] $[\psi(v), \psi(v') ] =\mathrm{Lie}(\phi)([v, v'])$, \quad $v, v' \in V $.
\end{itemize}
We let $\mathsf{HCP}_{\Bbbk}$ denote the category of Harish-Chandra pairs over $\Bbbk$.

This category $\mathsf{HCP}_{\Bbbk}$ is isomorphic to
the full subcategory of $(\operatorname{sHCP})_{\Bbbk}$ 
consisting of the objects $(G, \mathfrak{g})$ in which $G$ is an algebraic group.
To describe an explicit category isomorphism, let  
$(G, V) \in \mathsf{HCP}_{\Bbbk}$. 
Define $\mathfrak{g} := \mathrm{Lie}(G)\oplus V \in \mathsf{SMod}_{\Bbbk}$ with
$\mathfrak{g}_0 =\mathrm{Lie}(G)$,\ $\mathfrak{g}_1=V$.
Give to $\mathfrak{g}$ the bracket on $\mathrm{Lie}(G)$ as well as the structure
$[\ , \ ]$ of $(G,V)$, and define $[v, x] := v \triangleleft x$ for $v$, $x$ as in  
\eqref{eq:v_triangleleft_x}. Then $\mathfrak{g}$ turns into a Lie superalgebra. 
Retain the right $G$-module structure on $\mathfrak{g}_1=V$. 
One sees that $(G, V) \mapsto (G, \mathfrak{g})$ gives the desired category isomorphism. 
The inverse is given by
$(G, \mathfrak{g}) \mapsto (G, \mathfrak{g}_1)$, where 
the $\mathfrak{g}_1$ of the latter is given
the restricted super-bracket and the original $G$-module structure. 

Now, let $\mathbf{G} \in \mathsf{ASG}_{\Bbbk}$. Then $\mathbf{G}_{ev}$ is an algebraic group,
and the Lie superalgebra $\operatorname{Lie}(\mathbf{G})$ is finite-dimensional. 
Regard the odd component $\operatorname{Lie}(\mathbf{G})_1$ of the Lie superalgebra as 
the right $\mathbf{G}_{ev}$-module defined by \eqref{eq:vg}. Restrict the super-bracket on 
$\operatorname{Lie}(\mathbf{G})$ to the odd component, and give it to the pair 
$(\mathbf{G}_{ev}, \operatorname{Lie}(\mathbf{G})_1)$. Then the pair turns into a Harish-Chandra pair,
and it corresponds to $\Phi(\mathbf{G})$ in $(\operatorname{sHCP})_{\Bbbk}$ 
through the category isomorphism above. 
By Theorem \ref{thm:Gavarini's_equivalence} we have:

\begin{theorem}\label{thm:equivalence_over_field}
$\mathbf{G} \mapsto
(\mathbf{G}_{ev}, \operatorname{Lie}(\mathbf{G})_1)$ gives a category equivalence 
\[ 
\mathsf{ASG}_{\Bbbk} \overset{\approx}{\longrightarrow} \mathsf{HCP}_{\Bbbk}. 
\]  
\end{theorem}

Essentially the same result was already given in \cite{M3, MZ}; 
see Remark \ref{rem:coincidence_up_to_isom} for a subtle difference caused by choice of
dualities. 
As an advantage here, we have obtained an explicit quasi-inverse of the functor above,
which is essentially the same as $\Psi$ in Section \ref{subsec:reprove}. 
Therefore, every algebraic supergroup can be realized as $\mathbf{\Gamma}$ constructed 
in Section \ref{sec:construction}. 
This realization is useful when we discuss
group-theoretical properties of algebraic supergroups, 
as will be shown in the next subsection. 

\begin{rem}\label{rem:coincidence_up_to_isom}
A category equivalence between $\mathsf{ASG}_{\Bbbk}$ and $\mathsf{HCP}_{\Bbbk}$
is given by \cite[Theorem 6.5]{M3} and \cite[Theorem 3.2]{MZ}, 
which both reformulate the result \cite[Theorem 29]{M2} formulated in purely Hopf-algebraic terms.
Given $(G, V) \in \mathsf{HCP}_{\Bbbk}$, denote now it by $(G, V, [\ , \ ])$, indicating the structure. 
Replacing $[\ , \ ]$ with $-[\ , \ ]$, we still have $(G, V, -[\ , \ ]) \in \mathsf{HCP}_{\Bbbk}$.
Moreover, $(G, V, [\ , \ ])\mapsto (G, V, -[\ , \ ])$ gives an involutive category isomorphism 
$\mathsf{HCP}_{\Bbbk}\to \mathsf{HCP}_{\Bbbk}$. 
The equivalence given by Theorem 
\ref{thm:equivalence_over_field}, composed with the last isomorphism, coincides
with the equivalence cited above, as is seen from \eqref{eq:up_to_isom}. 
\end{rem}

\subsection{An application}\label{subsec:application}
Throughout in this subsection we let $\mathbf{G}\in \mathsf{ASG}_{\Bbbk}$,
and let $(G, V)$ be the associated Harish-Chandra pair. We suppose that $\mathbf{G}$ is realized as the
$\mathbf{\Gamma}$ which is constructed as in Section \ref{sec:construction} from 
$G$,\ $\mathfrak{g}:=\operatorname{Lie}(\mathbf{G})$,\ the canonical pairing 
$\mathfrak{g}_0 \times \mathcal{O}(G) \to \Bbbk$ and the 
right $G$-supermodule structure on $\mathfrak{g}$ defined 
by \eqref{eq:xg} and \eqref{eq:vg}. 

\begin{definition}\label{def:sub-pair} 
Let $(H, W)$ be a pair of closed subgroup $H \subset G$ and a sub-vector space $W \subset V$. 
We say that $(H,V)$ is a \emph{sub-pair} of the Harish-Chandra pair $(G, V)$, if 
\begin{itemize}
\item[(H1)]
$W$ is $H$-stable in $V$, and
\item[(H2)] 
$[W,W] \subset \operatorname{Lie}(H)$,
\end{itemize}
where $[\ , \ ]$ is the structure of $(G, V)$.
\end{definition}

If $\mathbf{H}$ is
a closed super-subgroup of $\mathbf{G}$, 
then the associated Harish-Chandra pair $(H, W)$, with the right $H$-module structure on $W$ 
as well as
the structure $[\ , \ ]$ forgotten, is a sub-pair of $(G, V)$. 
In this case we say that the sub-pair $(H,W)$ \emph{corresponds to} $\mathbf{H}$. 
The assignment $\mathbf{H}\mapsto (H,W)$ as above gives a bijection from
the set of all closed super-subgroups of $\mathbf{G}$ to the set of all sub-pairs of $(G,V)$. 

\begin{lemma}\label{lem:sub-pair}
Let $(H, W)$ be the sub-pair of $(G,V)$ corresponding to 
a closed super-subgroup $\mathbf{H}\subset \mathbf{G}$.
Given $v \in V$, the following are equivalent:
\begin{itemize}
\item[(i)] $v \in W$;
\item[(ii)] $e(a,v) \in \mathbf{H}(A)$ for arbitrary $A \in \mathsf{SAlg}_{\Bbbk}$ and $a \in A_1$;
\item[(iii)] $e(a, v) \in \mathbf{H}(A)$ for some $A \in \mathsf{SAlg}_{\Bbbk}$ and 
$a \in A_1$ with $a \ne 0$.
\end{itemize}
\end{lemma}
\begin{proof}
We only prove (iii) $\Rightarrow$ (i), since the rest is obvious.

Suppose that $e(a, v) \in \mathbf{H}(A)$ with $a \in A_1$, but $v \notin W$. 
Given an arbitrary basis $w_1,\dots,w_r$ of $W$, one can extend it, adding $v$ and others, 
to a basis $w_1,\dots, w_r, v, \dots$ of $V$. By Proposition \ref{prop:basis}, 
$e(a, v)$, being an element in $\mathbf{H}(A)$,
is expressed uniquely in the form
\begin{equation}\label{eq:two_expressions}
e(a,v)= h\, e(a_1,w_1)\dots e(a_r, w_r), 
\end{equation}
where $h \in H(A_0)$ and $a_i \in A_1$, $1 \le i \le r$. The cited proposition gives 
analogous expressions
of elements of $\mathbf{G}(A)$ which use the extended basis. Regarding \eqref{eq:two_expressions} as two
such expressions of one element, we have $a=0$. 
\end{proof}

Just as in the non-super situation we define as follows, and obtain the next lemma; see 
\cite[Part I, 2.6]{J}. 

Let $\mathbf{H}\subset \mathbf{G}$ be a closed super-subgroup.
The \emph{normalizer} $\mathcal{N}_{\mathbf{G}}(\mathbf{H})$
(resp., the \emph{centralizer} $\mathcal{Z}_{\mathbf{G}}(\mathbf{H})$) of $\mathbf{H}$ in $\mathbf{G}$
is the subgroup functor of $\mathbf{G}$ whose $A$-points consists of the elements
$g \in \mathbf{G}(A)$ such that for every $A \to A'$ in $\mathsf{SAlg}_{\Bbbk}$, the natural
image $g_{A'}$ of $g$ in $\mathbf{G}(A')$ normalizes (resp., centralizes) $\mathbf{H}(A')$. 

\begin{lemma}\label{lem:normal_center}
$\mathcal{N}_{\mathbf{G}}(\mathbf{H})$ and $\mathcal{Z}_{\mathbf{G}}(\mathbf{H})$
are closed super-subgroups of $\mathbf{G}$. Moreover, 
$\mathcal{N}_{\mathbf{G}}(\mathbf{H})$ (resp., $\mathcal{Z}_{\mathbf{G}}(\mathbf{H})$)
is the largest closed super-subgroup of $\mathbf{G}$ whose $A$-points normalize (resp., centralize)
$\mathbf{H}(A)$ for every $A \in \mathsf{SAlg}_{\Bbbk}$. 
\end{lemma}

Let $\mathbf{H}\subset \mathbf{G}$ be a closed super-subgroup, and let $(H,W)$ be the corresponding 
sub-pair of $(G, V)$. 

Recall that the \emph{stabilizer} $\mathrm{Stab}_G(W)$ 
(resp., the \emph{centralizer} $\mathrm{Cent}_G(W)$) of $W$ in $G$
is the largest closed subgroup of $G$ that makes $W$ into a module
(resp., a trivial module) over it.

Let $\rho_H : V \to \mathcal{O}(H)\otimes V$ denote the left $\mathcal{O}(H)$-comodule
structure on $V$ corresponding to the restricted right $H$-module structure on $V$. 
Define
\[
\mathrm{Inv}_H(V/W):=\{ v \in V \mid \rho_H(v)-1 \otimes v\in \mathcal{O}(H)\otimes W\}.
\]
This is the largest $H$-submodule of $V$ including $W$ whose quotient $H$-module
by $W$ is trivial. The definition makes sense, replacing $W$ with any $H$-submodule, say $U$, of $V$. 
We will use $\mathrm{Inv}_H(V)=\mathrm{Inv}_H(V/0)$ when $U=0$.

When $L=\mathrm{Lie}(H)$ or $0$, we define
\[
(L:W):=\{ v \in V \mid [v, W]\subset L\},
\]
where $[\ , \ ]$ is the structure of $(G,V)$. 

\begin{theorem}\label{thm:normal_central}
Let $\mathbf{H} \subset \mathbf{G}$ and $(H,W) \subset (G,V)$ be as above. 
\begin{itemize}
\item[(1)] The sub-pair of $(G,V)$ corresponding to $\mathcal{N}_{\mathbf{G}}(\mathbf{H})$ is
\[
( \mathcal{N}_{G}(H)\cap \mathrm{Stab}_G(W),\ \mathrm{Inv}_H(V/W) \cap (\mathrm{Lie}(H) : W) ) .
\]
\item[(2)] The sub-pair of $(G,V)$ corresponding to $\mathcal{Z}_{\mathbf{G}}(\mathbf{H})$ is
\[
( \mathcal{Z}_{G}(H)\cap \mathrm{Cent}_G(W),\ \mathrm{Inv}_H(V) \cap (0 : W) ).
\]
\end{itemize}
\end{theorem}
\begin{proof}
In each part let us denote by $(K,Z)$ the desired sub-pair. 

(1)\
First, we prove
\begin{equation}\label{eq:one_direction_inclusions}
K \subset \mathcal{N}_{G}(H)\cap \mathrm{Stab}_G(W),\quad 
Z \subset \mathrm{Inv}_H(V/W) \cap (\mathrm{Lie}(H) : W). 
\end{equation}

Note that $K$ normalizes $\mathbf{H}$ in $\mathbf{G}$.
Then it follows that $K$ normalizes $H=\mathbf{H}_{ev}$ in $G=\mathbf{G}_{ev}$, 
whence $K \subset \mathcal{N}_{G}(H)$.
It also follows that the right $G$-supermodule structure on $\operatorname{Lie}(\mathbf{G})$, restricted
to a right $K$-supermodule structure, 
stabilizes $\operatorname{Lie}(\mathbf{H})$, whence $K \subset \mathrm{Stab}_G(W)$.

Since $[\operatorname{Lie}(\mathbf{H}), \operatorname{Lie}(\mathcal{N}_{\mathbf{G}}(\mathbf{H}))]
\subset \operatorname{Lie}(\mathbf{H})$, we have $[W ,Z]\subset \operatorname{Lie}(H)$, whence
$Z \subset (\mathrm{Lie}(H) : W)$. 

To prove $Z \subset \mathrm{Inv}_H(V/W)$, choose $z \in Z$. We may suppose $z\notin W$.
Let $A = \mathcal{O}(H)\otimes \wedge(\tau)$ with $\tau$ an odd variable. We have an
$A$-point $e(\tau, z)$ of $\mathcal{N}_{\mathbf{G}}(\mathbf{H})$ by Lemma \ref{lem:sub-pair}. 
Given a basis $w_1,\dots, w_r$ of $W$, we extend it, adding $z$ and others, to a basis
$w_1,\dots, w_r, z, u_1,\dots, u_s$ of $V$.
Present $\rho_H(z)$ as
\[ 
\rho_H(z) = \sum_{i=1}^r a_i \otimes w_i + b \otimes z + \sum_{i=1}^s c_i \otimes u_i \in \mathcal{O}(H)\otimes V. 
\]
Let $h \in H(A_0)$ be $\operatorname{id}_{\mathcal{O}(H)}$. Using Lemmas \ref{lem:relations} and \ref{lem:egfg},
one computes
\begin{equation}\label{eq:e_tau_z}
\begin{aligned}
&e(\tau, z) \, h\, e(\tau, z)^{-1} =\\ 
&\quad h\, e(a_1\tau, w_1)\dots e(a_r\tau, w_r)\, e((b-1)\tau, w_r)\,
e(c_1 \tau, u_1)\dots e(c_s \tau, u_s).
\end{aligned}
\end{equation}
Since this is contained in $\mathbf{H}(A)$, it follows by the same argument as proving 
Lemma \ref{lem:sub-pair} 
that $b=1$ and $c_i=0$, $1 \le i \le s$, whence
$Z \subset \mathrm{Inv}_H(V/W)$. 
We have thus proved \eqref{eq:one_direction_inclusions}. 

Next, to prove the converse inclusions, choose 
$\phi : A \to A'$ from $\mathsf{SAlg}_{\Bbbk}$.

Let $g$ be an $A$-point of $\mathcal{N}_{G}(H)\cap \mathrm{Stab}_G(W)$. Then $g_{A'}$ normalizes $H(A')$. 
Given $a\in A'_1$ and $w \in W$, we have
\[
e(a, w)^{g_{A'}} = 1 \otimes 1 + a\, w^g \in \mathbf{H}(A'), 
\]
and the same result with $g$ replaced by $g^{-1}$ holds. 
This proves $g \in K(A)$.  

Let $v \in \mathrm{Inv}_H(V/W) \cap (\mathrm{Lie}(H) : W)$ and $0\ne a \in A_1$. 
To see that $v \in Z$, we wish to prove, using Lemma \ref{lem:sub-pair}, that
$e(a, v)$ is an $A$-point of $\mathcal{N}_{\mathbf{G}}(\mathbf{H})$. Note that the $A'$-point
$e(a, v)_{A'}$ of its image is $e(\phi(a), v)$. 
Given $h \in H(A')$, the
same argument as proving \eqref{eq:e_tau_z} shows 
$e(a,v)_{A'}\, h\, e(a,v)_{A'}^{-1} \in \mathbf{H}(A')$, since $v^h- v \in W_{A'_0}$. 
Given $w \in W$ and $b \in A'_1$, we see by Lemma \ref{lem:relations} (i) that
\[ 
e(a, v)_{A'}\, e(b, w)\, e(a,v)_{A'}^{-1} = i(f(-\phi(a)b, [v, w]))\, e(b,w) \in \mathbf{H}(A'), 
\]
since $[v,w]\in \operatorname{Lie}(H)$. The last two conclusions prove the desired result. 

(2)\
We only prove
\[
K \subset \mathcal{Z}_{G}(H)\cap \mathrm{Cent}_G(W),\quad 
Z \subset \mathrm{Inv}_H(V) \cap (0 : W).
\]
The converse inclusions follow by modifying slightly the second half of the proof of Part 1. 

Since $K$ centralizes $\mathbf{H}$ in $\mathbf{G}$.
it follows that $K$ centralizes $H$ in $G$, 
whence $K \subset \mathcal{Z}_{G}(H)$.
It also follows that the restricted right $K$-supermodule structure on $\operatorname{Lie}(\mathbf{G})$
centralizes $\operatorname{Lie}(\mathbf{H})$, whence $K \subset \mathrm{Cent}_G(W)$.

Since $[\operatorname{Lie}(\mathbf{H}), \operatorname{Lie}(\mathcal{Z}_{\mathbf{G}}(\mathbf{H}))]
= 0$, we have $[W ,Z]=0$, whence
$Z \subset (0 : W)$. The argument which proved $Z \subset \mathrm{Inv}_H(V/W)$ above, modified with 
$W$ replaced by $0$, shows $Z \subset \mathrm{Inv}_H(V)$. 
\end{proof}

Suppose that $\mathbf{G}=\mathbf{H}$, and so $G=H$, $V=W$. Then Part 2 above reads:

\begin{corollary}\label{cor:center}
Let $\mathbf{G}$ and $(G,V)$ be as above. The sub-pair of $(G,V)$ corresponding to the center
$\mathcal{Z}(\mathbf{G})= \mathcal{Z}_{\mathbf{G}}(\mathbf{G})$ of $\mathbf{G}$ is
\[
( \mathcal{Z}(G)\cap \mathrm{Cent}_G(V),\ \mathrm{Inv}_G(V) \cap (0 : V) ).
\]
\end{corollary}

The algebraic-group component of this sub-pair was obtained by \cite[Proposition 7.1]{MZ}
in an alternative way. 
\bigskip

\noindent
\emph{Note added in revision.}\quad In present paper the term ``algebraic (super)group" is used
in a restricted sense to indicate \emph{affine} algebraic (super)groups 
(see Section \ref{subsec:ASG_and_HCP}), and
the category equivalence theorem, Theorem \ref{thm:equivalence_over_field}, concerns 
affine algebraic supergroups.  
Very recently Alexandr Zubkov and the first-named author generalized this theorem
to not necessarily affine, algebraic supergroups over a field of characteristic $\ne 2$,
and showed that the results in Section \ref{subsec:application} above hold for those supergroups, more generally.
Details will be contained in a forthcoming joint paper.

\section*{Acknowledgments}
The first-named author was supported by JSPS Grant-in-Aid for Scientific Research (C)~~26400035. 
The second-named author was supported by Grant-in-Aid for JSPS Fellows 26E2022.
The authors thank Alexandr Zubkov for his valuable comments on an earlier version of the paper,
which include a corrected proof of Proposition \ref{prop:PhiA} (2).

\end{document}